\documentclass[11pt]{article}
\usepackage[utf8]{inputenc}
\usepackage{amsmath,amssymb}
\usepackage{amsthm,amssymb,amsmath}
\usepackage{enumitem,mathtools}
\usepackage{url}
\usepackage{fullpage}
\usepackage{hyperref}
\usepackage{graphicx}
\usepackage{array}
\usepackage{multirow}

\usepackage{longtable}
\usepackage{tikz}
\usepackage{bbm}
\usepackage{bm}
\usepackage{comment}
\usepackage{float}

\makeatletter
\newcommand*{\transpose}{%
  {\mathpalette\@transpose{}}%
}
\newcommand*{\@transpose}[2]{%
  \raisebox{\depth}{$\m@th#1\intercal$}%
}

\newtheorem{theorem}{Theorem}[section]

\newtheorem{example}[theorem]{Example}

\newtheorem{corollary}[theorem]{Corollary}
\newtheorem{lemma}[theorem]{Lemma}

\newtheorem{remark}[theorem]{Remark}
\newtheorem{question}[theorem]{Question}

\usepackage{pgf}
\usepackage{tikz}
\usetikzlibrary{arrows,decorations.markings,shapes}
\tikzstyle{vertex}=[circle, draw, fill=black, inner sep=0pt, minimum width=6pt]
\tikzstyle{vert}=[circle, draw=black, fill=white, inner sep=0pt, minimum width=6pt]
\tikzstyle{zc}=[circle, draw, fill=black, inner sep=0pt, minimum width=6pt]
\tikzstyle{pc}=[circle, draw=black, inner sep=1pt, minimum width=10pt, font=\tiny] 
\tikzstyle{nc}=[circle, draw=black, inner sep=1pt, minimum width=10pt, font=\tiny] 
\tikzstyle{pedge}=[draw,-]
\tikzstyle{wedge}=[draw,-,postaction={decorate}, decoration={markings,mark = at position 0.55 with {\arrow{stealth} } }]
\tikzstyle{dpedge}=[draw,-,postaction={decorate}]
\tikzstyle{wwedge}=[draw,-,postaction={decorate}, decoration={markings,mark = at position 0.5 with {\arrow{stealth} }, mark = at position 0.6 with {\arrow{stealth} } }]
\tikzstyle{wwedge2}=[draw,-,postaction={decorate}, decoration={markings,mark = at position 0.45 with {\arrow{stealth} }, mark = at position 0.65 with {\arrow{stealth} } }]
\tikzstyle{wwwedge}=[draw,-,postaction={decorate}, decoration={markings,mark = at position 0.65 with {\arrow{stealth} },mark = at position 0.45 with {\arrow{stealth} }, mark = at position 0.55 with {\arrow{stealth} } }]
\tikzstyle{wwwedge2}=[draw,-,postaction={decorate}, decoration={markings,mark = at position 0.7 with {\arrow{stealth} },mark = at position 0.4 with {\arrow{stealth} }, mark = at position 0.55 with {\arrow{stealth} } }]
\tikzstyle{wwwwedge}=[draw,-,postaction={decorate}, decoration={markings,mark = at position 0.7 with {\arrow{stealth} },mark = at position 0.4 with {\arrow{stealth} }, mark = at position 0.5 with {\arrow{stealth} }, mark = at position 0.6 with {\arrow{stealth} } }]
\tikzstyle{wwwnedge}=[draw,densely dashed,postaction={decorate}, decoration={markings,mark = at position 0.65 with {\arrow{stealth} },mark = at position 0.45 with {\arrow{stealth} }, mark = at position 0.55 with {\arrow{stealth} } }]
\tikzstyle{wwnedge}=[draw,densely dashed,postaction={decorate}, decoration={markings,mark = at position 0.5 with {\arrow{stealth} }, mark = at position 0.6 with {\arrow{stealth} } }]
\tikzstyle{wnedge}=[draw,densely dashed,postaction={decorate}, decoration={markings,mark = at position 0.55 with {\arrow{stealth} } }]
\tikzstyle{wnedge2}=[draw,densely dashed,postaction={decorate}, decoration={markings,mark = at position 0.6 with {\arrow{stealth} } }]
\tikzstyle{dnedge}=[draw,densely dashed,postaction={decorate}]
\tikzstyle{nedge}=[draw,densely dashed]
\tikzstyle{weight2}= [draw=white, fill=white, font=\scriptsize]
\tikzstyle{weight}= [font=\scriptsize]
\tikzstyle{empty}=[circle, draw=white, inner sep=2pt, fill=white, minimum width=4pt]
\tikzstyle{ghost}=[circle, draw=black, inner sep=1pt, style=densely dashed, minimum width=6pt, font=\tiny]
\tikzstyle{ghostc}=[circle, draw=black, inner sep=1pt, style=densely dashed, minimum width=10pt, font=\tiny]
\tikzstyle{dedge}=[draw,very thick,dotted]

\title{Constructions of $t$-designs from \\ weighing matrices and association schemes}

\author{ Gary Greaves\thanks{G.G. was supported by  the Singapore
Ministry of Education Academic Research Fund;  grant numbers: RG18/23 (Tier 1) and MOE-T2EP20222-0005 (Tier 2). 
}\\
  School of Physical and Mathematical Sciences, \\
  Nanyang Technological University, \\
   21 Nanyang Link, Singapore 637371\\
  {\tt gary@ntu.edu.sg}
\and
 Sho Suda\thanks{S.S. was supported by JSPS KAKENHI; grant number: 22K03410. 
 } \\
 Department of Mathematics, \\  National Defense Academy of Japan, \\ Yokosuka, 239-8686, Japan \\
 \tt{ssuda@nda.ac.jp}
}
\date{\today}

\begin{document}

\maketitle
\begin{abstract}
    We provide a method to construct $t$-designs from weighing matrices and association schemes.
    One instance of our method can produce a $3$-design from any (symmetric or skew-symmetric) conference matrix, 
    thereby providing a partial answer to a question of Gunderson and Semeraro JCTB 2017.
    We explore variations of our method on some matrices that satisfy certain combinatorial restrictions.
    In particular, we show that there exist various infinite families of partially balanced incomplete block designs with block size four on the binary Hamming schemes and the $3$-class association schemes attached to symmetric designs, and regular pairwise balanced designs with block sizes three and four. 
\end{abstract}

\section{Introduction}

\subsection{Background}

Let $t$, $v$, $k$, and $\lambda$ be positive integers. 
A \textbf{$t$-$(v,k,\lambda)$ design} is a $k$-uniform hypergraph $(\mathfrak X,\mathfrak B)$ with the property that any set of $t$ vertices from the $v$-element set $\mathfrak X$ appear together in precisely $\lambda$ hyperedges of $\mathfrak B$.
One of the main purposes of research on design theory is to classify $t$-$(v,k,\lambda)$ designs, i.e., for given parameters $t$, $v$, $k$, and $\lambda$, construct the corresponding design, enumerate the possible constructions and determine the equivalence classes of $t$-$(v,k,\lambda)$ designs up to isomorphism, or show that no such design can exist. 
In widely celebrated recent work of Keevash~\cite{keevash} and subsequently Glock-K\"uhn-Lo-Osthus~\cite{gklo}, ingenious probabilistic techniques were employed to show that $t$-$(v,k,\lambda)$ designs exist asymptotically (when $v$ is large enough compared to $t$, $k$, $\lambda$ and subject to the necessary divisibility conditions).
However, it remains an important open problem to provide explicit constructions.

Constructions of designs have been investigated from perspectives of various disciplines, such as finite geometry \cite{Ball}, character theory \cite{P,S}, and coding theory with the Assmus-Mattson theorem \cite{AM}. 
We contribute to the construction of designs by proposing a new method of construction using square matrices satisfying certain algebraic conditions. 
Typically, the matrices $A$ to which our method can be applied have entries belonging to the set $\{-1,0,1\}$ and satisfy $A A^\transpose = w I$ for some positive integer $w$.
Such matrices are known as \emph{weighing matrices} in the literature.
We will see that our method can also be applied to some variations of weighing matrices as well as to association schemes.
 
Motivated by a classical problem of Frankl and F\"{u}redi~\cite{FF}, Gunderson and Semeraro~\cite[Theorem~11]{GunSem}, 
construct $3$-$(q+1,4,(q+1)/4)$ designs by collecting $4$-vertices subtournaments so-called diamonds in Paley tournaments on $q$ vertices where $q$ is a prime power congruent to $3$ modulo $4$. 
Here, diamonds are tournaments on four vertices obtained by adding a dominating or dominated vertex to the $3$-cycle.  
\textbf{}
It was observed by Belcouche et al.~\cite{BBLZ} that the construction of Gunderson and Semeraro corresponds to taking $4$-subsets of $\{1,\dots,q+1\}$ for which the corresponding $4 \times 4$ principal minor of a matrix representation of the Paley tournament is $9$.
Accordingly, for a $v\times v$ matrix $A$, we define the hypergraph $\mathfrak H_A(v,k,a)=(\{1,\dots,v\},\mathfrak B_A(v,k,a))$ by 
    \[\mathfrak B_A(v,k,a) :=\left\{k\text{-subsets of }\{1,\dots,v\}\text{ whose corresponding principal minor of } A \text{ is equal to } a    \right\}.\]
Let $q\equiv3\pmod{4}$ be a prime power and $S$ be a matrix whose rows and columns are indexed by the elements of the finite field $\mathbb{F}_q$ of order $q$ and the $(x,y)$ entry of $S$ is $0$ if $x=y$, $1$ if $x-y$ is a nonzero quadratic residue modulo $q$, and $-1$ if $x-y$ is a non-quadratic residue modulo $q$.
Append a row and column to $S$ to form the $(q+1) \times (q+1)$ matrix $A$ where the entries of the last row and column are equal to $1$ except for the diagonal entry, which is $0$.
The hypergraph $\mathfrak H_{A}(q+1,4,9)$ is a $3$-$(q+1,4,(q+1)/4)$ design.

\subsection{Main results and their consequences}

Our purpose is to find conditions on the matrix $A$ and parameters $v$, $k$, and $a$ that guarantee the hypergraph $\mathfrak H_{A}(v,k,a)$ is a $t$-$(v,k,\lambda)$ design.
One of our main results (see Theorem~\ref{thm:des}, below) is as follows.
\begin{theorem}[Synopsis of Theorem~\ref{thm:des}]
\label{thm:main1}
    Let $A$ be a $v \times v$ complex matrix that has precisely two principal $k \times k$ minors $a$ and $b$.
    Suppose, for each $i \in \{0,1,\dots,t\}$, the coefficient of $x^{v-k-i}$ of the characteristic polynomial of each principal $(v-i) \times (v-i)$ submatrix of $A$ is constant.
    Then the hypergraph $\mathfrak H_A(v,k,a)$ is a $t$-$(v,k,\lambda)$ design. 
\end{theorem}
 
The $\lambda$ in Theorem~\ref{thm:main1} can be expressed in terms of the spectrum of $A$ and $a$ and $b$.
Such an expression is given explicitly in Theorem~\ref{thm:des}, below.
In Section~\ref{sec:des}, we give examples to illustrate the utility of Theorem~\ref{thm:des}.
In particular, Theorem~\ref{thm:des} can be used to provide a partial answer to a question of
Gunderson and Semeraro~\cite[Question~25]{GunSem} who asked 
\begin{question}
\label{qsn:gs}
    For which $n\equiv0\pmod{4}$ does there exist a $3$-$(n,4,n/4)$ design with the property that every set of $5$ vertices contains either $0$ or $2$ hyperedges?
\end{question}
As we show in Example~\ref{ex:skew3}, it follows from Theorem~\ref{thm:des} that such a $3$-$(n,4,n/4)$ design exists whenever there exists a skew-symmetric conference matrix of order $4n$.
It is conjectured in \cite{S78} that skew-symmetric conference matrices of order $4n$ exist for all $n \in \mathbb N$.
The validity of this conjecture would thus completely answer the question of Gunderson and Semeraro, and would furthermore provide a partial answer to \cite[Problem~1]{FF}.

In addition to the above, Theorem~\ref{thm:des} provides constructions of infinite families of  $3$-designs from symmetric conference matrices; infinite families of $2$-designs from doubly regular tournaments, real equiangular tight frames, and Hermitian complex Hadamard matrices over the third roots of unity.

We continue our investigation by extending our main theorem to produce partially balanced incomplete block designs from square matrices together with a symmetric association scheme (see Theorem~\ref{thm:pbibd} and Theorem~\ref{thm:pbibd2}).
We apply Theorem~\ref{thm:pbibd} and Theorem~\ref{thm:pbibd2} to various large families of matrices and symmetric association schemes to provide infinite families of partially balanced block designs.
In particular, we obtain infinite families of partially balanced block designs from distance regular graphs, skew-Bush type Hadamard matrices and group divisible designs, signed hypercubes and binary Hamming schemes, and balanced generalised weighing matrices. 
Lastly, we show how Hadamard matrices can be used to construct infinite families of pairwise balanced designs (see Theorem~\ref{thm:hmpbd}).


\subsection{Organisation}

The organisation of this paper is as follows. 
In Section~\ref{sec:des}, we prove our main theorem, Theorem~\ref{thm:des}, and exhibit infinite families of $t$-designs for $t=2$
and $3$. 
Here, we provide a partial answer to Question~\ref{qsn:gs} and further demonstrate the utility of Theorem~\ref{thm:des} with numerous examples of its application.
In Section~\ref{sec:pbibd}, 
we show how to apply a variation of our method to an element of the Bose-Mesner algebra in a symmetric association scheme in order to obtain a partially balanced incomplete block design (PBIBD). 
A significant consequence is the construction of PBIBDs of hyperedge size three or four for any symmetric association scheme whose parameters are determined from the spectral data of the matrix used. 
We exhibit an infinite family of PBIBDs on the binary Hamming association scheme using a signed adjacency matrix of the Hamming graph. 
We also exhibit an infinite family of PBIBDs on a symmetric $3$-class association scheme using a balanced generalised weighing matrix. 
In Section~\ref{sec:pbd}, as a further application of the results in Section~\ref{sec:pbibd}, we construct an infinite family of regular pairwise balanced designs from a family of Hadamard matrices. 
The paper culminates with Section~\ref{sec:open}, where we list relevant open problems that emerged from our investigations.

\subsection{Notation}

\label{sec:notnprelim}

Throughout, for positive integers $v$ and $k$, define $[v]=\{1,\ldots,v\}$ and $\binom{[v]}{k}=\{\alpha \subset [v]  \; : \; |\alpha|=k\}$. 
Let $A$ be a $v\times v$ matrix whose rows and columns are indexed by $[v]$. 
For a subset $\alpha \subset[v]$, we denote by $A[\alpha]$ the principal submatrix of $A$ whose rows and columns are indexed by $\alpha$. 
We denote the complement of $\alpha$ by $\overline{\alpha}=[v]\setminus \alpha$ (the containing set $[v]$ should be clear from context).
Denote by $J_n$, $O_n$, and $I_n$ the all-ones matrix, the all-zeros matrix, and the identity matrix of order $n$.
We write $J$, $O$, and $I$, instead of $J_n$, $O_n$, and $I_n$ respectively if the order of the matrix is clear from context.

Define $D_A(k)$ to be the set of all $k\times k$ principal minors of $A$, i.e., $D_A(k) := \left \{ \det(A[\alpha]) \; : \; \alpha \in \binom{[v]}{k} \right \}$.
For $v \in \mathbb N$, $k \in \{0,1,\dots,v\}$, $d \in D_A(k)$, and $\beta \subset [v]$, we define $\lambda_A(\beta,k,d)$ and $\mu_A(\beta,k,d)$ by 
\begin{align*}
\lambda_A(\beta,k,d)& :=\left|\left\{\alpha \in \binom{[v]}{k} \; : \; \beta \subset \alpha, \; \det (A[\alpha])=d\right\}\right|,\\
\mu_A(\beta,k,d)& :=\left|\left\{\alpha \in \binom{[v]}{k} \; : \; \beta \cap \alpha=\emptyset, \; \det (A[\alpha])=d\right\}\right|. 
\end{align*}

For $k \in \{0,1,\dots,v\}$, a $v\times v$ complex matrix $A$, and a subset $\alpha \subset [v]$, define $c_{A}(\alpha,k)$ to be the coefficient of $x^{v-k-|\alpha|}$ in $\det(xI-A[\overline{\alpha}])$.
For a set $\mathfrak A$ of subsets of $[v]$, define
\[
C_{A}(\mathfrak A,k) := \left \{ c_{A}(\alpha,k) \; : \; \alpha \in \mathfrak A \right \}.
\]

A \textbf{hypergraph} is a pair $(V,E)$ where $V$ is its set of \textit{vertices} and $E$ is its set of \textit{hyperedges}, which consists of subsets of $V$.
A hypergraph $\mathfrak H$ is called $k$-\textbf{uniform} if each hyperedge has cardinality $k$.
For a $v\times v$ matrix $A$ and $a\in D_A(k)$, we define a hypergraph $\mathfrak H_{A}(v,k,a)=([v],\mathfrak B_A(v,k,a))$ by 
  \[\mathfrak B_A(v,k,a):=\left\{\alpha\in \binom{[v]}{k} \; : \;  \det (A[\alpha])=a \right\}.\]
An (undirected) \textbf{graph} is defined as a $2$-uniform hypergraph $(V,E)$, where $E$ is said to be a set of \textbf{edges}. 
The \textbf{adjacency matrix} $A$ of an undirected graph $G = (V,E)$ is a $\{0,1\}$-matrix whose rows and columns are indexed by $V$ such that the $(u,v)$-entry of $A$ is $1$ if and only if $\{u,v\} \in E$.
A \textbf{directed graph} (or digraph) is a pair $(V,E)$ where $V$ is its set of vertices and its set of \textit{arcs} $E$ is a subset of the Cartesian square $V \times V$ that does not contain any element of the form $(x,x)$.
A \textbf{tournament} is a directed graph $(V,E)$ where $|\{ (u,v),(v,u)\} \cap E| = 1$ for each pair of distinct vertices $u \ne v$.
The \textbf{adjacency matrix} $A$ of a directed graph $G = (V,E)$ is a $\{0,1\}$-matrix whose rows and columns are indexed by $V$ such that the $(u,v)$-entry of $A$ is $1$ if and only if $(u,v) \in E$.

For a family of subsets $\mathfrak B$ of $[v]$ and a subset $\beta\subset[v]$, define 
\begin{align*}
\lambda_\mathfrak B(\beta)&:=|\{\alpha \in \mathfrak B  \; : \;  \beta \subset \alpha\}|,\\
\mu_\mathfrak B(\beta)&:=|\{\alpha \in \mathfrak B  \; : \;  \beta \cap \alpha=\emptyset\}|. 
\end{align*}
Note that $\lambda_\mathfrak B(\emptyset) = \mu_\mathfrak B(\emptyset)=|\mathfrak B|$. 


\subsection{Preliminaries}

In this subsection, we collect some preliminary results about determinants of principal submatrices that we will use in the later sections.
The following two lemmas provide a convenient way to express the characteristic polynomial of a principal submatrix, which we make use of in our main results.
The first lemma is a classical identity due to Jacobi (see \cite[Page 21]{Jacobi1} or \cite[Page 52]{Jacobi2}), which can be seen as a consequence of the Schur complement.
\begin{lemma}[Jacobi]
\label{lem:decaenjacobi}
    Let $A$ be a complex square matrix of order $v$ and let $\alpha \subset [v]$.
    Then
    \[
    \det(xI - A[\overline{\alpha}]) = \det(xI-A) \det( (xI-A)^{-1}[\alpha] ).
    \]
\end{lemma}

Denote by $e_j(x_1,\dots,x_s)$ the $j$-th elementary symmetric polynomial in $\mathbb R[x_1,\dots,x_s]$.

\begin{lemma}
\label{lem:Jacobi2}
Let $A$ be a $v \times v$ complex matrix with $\det(xI-A) = \prod_{i=1}^s(x-\theta_i)^{m_i}$ for distinct $\theta_i$.
Suppose that $A$ is diagonalisable.
Then, for each $\alpha \subset [v]$, we have
\[
\det(xI-A[\overline{\alpha}]) =\prod_{i=1}^s(x-\theta_i)^{m_i-|\alpha|} \cdot\det \left (\sum_{i=1}^s \left(\sum_{j=0}^i (-1 )^j e_j(\theta_1,\ldots,\theta_s)x^{i-j}\right)A^{s-i}[\alpha] 
 \right ).
\]
\end{lemma}
\begin{proof}
Since 
\[
A^s-e_1(\theta_1,\ldots,\theta_s)A^{s-1}+e_2(\theta_1,\ldots,\theta_s)A^{s-2}+\cdots+(-1)^s e_s(\theta_1,\ldots,\theta_s)I= O,
\]
it is routinely shown that
\[
(xI-A)^{-1} = \frac{\sum_{i=1}^s (\sum_{j=0}^i (-1 )^j e_j(\theta_1,\ldots,\theta_s)x^{i-j})A^{s-i}}{\prod_{i=1}^s (x-\theta_i)}.
\]
The lemma then follows from Lemma~\ref{lem:decaenjacobi}.
\end{proof}

Since this case is used repeatedly below, we exhibit the case when $s=2$ in Lemma~\ref{lem:Jacobi2} explicitly.
\begin{corollary}
\label{cor:Jacobi22}
Let $A$ be a $v \times v$ complex matrix with $\det(xI-A) = (x-\lambda)^{m_\lambda}(x-\mu)^{m_\mu}$ for $\lambda \neq \mu$.
Suppose that $A$ is diagonalisable.
Then, for each $\alpha \subset [v]$, we have
\[
\det(xI-A[\overline{\alpha}]) = (x-\lambda)^{m_\lambda-|\alpha|}(x-\mu)^{m_\mu-|\alpha|}\det \left ( A[\alpha]+(x-\lambda-\mu)I \right ).
\]
\end{corollary}

The following lemma (see, for example, \cite[Eq.(1,2,13), Page 53]{HJ}) is crucial in this paper. 
\begin{lemma}\label{lem:coef}
Let $v$ be a positive integer, $k \in \{0,1,\dots,v\}$, $A$ be a $v\times v$ complex matrix, and let $\alpha$ be a subset of $[v]$. 
Then
\begin{equation*}
    c_A(\alpha,k) = (-1)^{k}\sum_{\beta \in \binom{[v]\setminus \alpha}{k} }\det (A[\beta]).
\end{equation*}
\end{lemma}

Our final preliminary result is a corollary of Lemma~\ref{lem:coef}, which provides two useful equations.

\begin{corollary}\label{cor:mu}
Let $v$ be a positive integer, $k \in \{0,1,\dots,v\}$, $A$ be a $v\times v$ complex matrix, and let $\alpha$ be a subset of $[v]$. 
Then
\begin{enumerate}
    \item $\displaystyle \binom{v-|\alpha|}{k} = \sum_{d\in D_A(k)}\mu_A(\alpha,k,d)$;
    \item  
     $\displaystyle c_A(\alpha,k)    =(-1)^{k}\sum_{d\in D_A(k)}d\mu_A(\alpha,k,d).$

\end{enumerate}
\end{corollary}

\section{Designs constructed from square matrices}\label{sec:des}

In this section, we introduce a tool (Theorem~\ref{thm:des}) that enables us to produce designs from certain matrices.
We also provide examples of designs that we are able to produce in this way.
At this point, we note that one can construct a $t$-$(v,k,\lambda)$ design $(\mathfrak X,\mathfrak B)$ in a somewhat trivial way setting $\mathfrak B$ to be the set of all $k$-subsets of $\mathfrak X$.
Such a design is aptly called \textbf{trivial}.

\subsection{Hypergraphs from square matrices}
Before we establish our main result, we require the following technical lemma.

\begin{lemma}
    \label{lem:lambda}
        Let $v$ and $t$ be positive integers and $\mathfrak B$ be a family of subsets of $[v]$. 
    For $\beta\subset [v]$, we have
    \begin{align*}
\lambda_\mathfrak B(\beta)=\sum_{\gamma \subset \beta}(-1)^{|\gamma|} \mu_\mathfrak B(\gamma).
\end{align*}
Suppose that, for each $i \in \{0,1,\dots,t\}$, there exists $\mu_i$ such that $\mu_i = \mu_\mathfrak B(\beta)$ for all $\beta \in \binom{[v]}{i}$. 
Then, for each $i \in [t]$ and $\beta \in \binom{[v]}{i}$, we have
\begin{align*}
    \lambda_\mathfrak B(\beta)=\sum_{j=0}^i (-1)^j \binom{i}{j}\mu_j.
\end{align*}
\end{lemma}
\begin{proof}
The proof of this fact is obtained by modifying the proof of \cite[Theorem~9.7]{Stinson} as follows. 
For any subset $\gamma\subset \beta$, it is clear that 
$$
\left|\bigcap_{z\in \gamma}\{\alpha \in \mathfrak B  \; : \;  z\not\in \alpha\}\right|=\left|\{\alpha \in \mathfrak B  \; : \;  \alpha \cap \gamma=\emptyset\}\right|=\mu_\mathfrak B(\gamma).$$ 
Then the principle of Inclusion-Exclusion asserts that 
\begin{align*}
    |\{\alpha\in \mathfrak B  \; : \;  \beta\subset \alpha\}|
    =\left|\mathfrak B\setminus\left(\bigcup_{z\in \beta}\{\alpha \in \mathfrak B  \; : \;  z\not\in \alpha\}\right)\right|
    =\sum_{\gamma\subset \beta}(-1)^{|\gamma|}\left|\bigcap_{z\in \gamma}\{\alpha \in \mathfrak B  \; : \;  z\not\in \alpha\}\right|, 
\end{align*}
which shows that 
$$
\lambda_\mathfrak B(\beta)=\sum_{\gamma\subset \beta}(-1)^{|\gamma|} \mu_\mathfrak B(\gamma).
$$
The remaining assertions follow from this equation. 
\end{proof}


The following is the main theorem of this section. 
\begin{theorem}\label{thm:des}
Let $t$, $v$, and $k$ be positive integers. 
Let $A$ be a $v\times v$ complex diagonalisable matrix such that 
\begin{enumerate}
    \item $D_A(k)=\{a,b\}$, where $a\neq b$, and
    \item $C_A\left(\binom{[v]}{i},k\right) = \{c_i\}$ for each $i \in \{0,1,\dots,t\}$.
\end{enumerate}
Then the hypergraph $\mathfrak H_{A}(v,k,a)$ is a $t$-$(v,k,\lambda)$ design, where 
\[
\lambda=\frac{(-1)^k\binom{k}{t}}{(a-b)\binom{v}{t}}c_A(\emptyset,k)-\frac{b}{a-b}\binom{v-t}{k-t}.
\]
\end{theorem}
\begin{proof}
Let $\mathfrak B = \mathfrak B_A(v,k,a)$.
First we claim that, for each $i \in \{0,1,\dots,t\}$, there exists $\mu_i$ such that $\mu_\mathfrak B(\beta) = \mu_i$ for each $\beta \in \binom{[v]}{i}$. 
By Corollary~\ref{cor:mu} 
together with assumption (1), we obtain 
\begin{align*}
    c_i = c_A(\beta,k)=(-1)^{k}\left((a-b)\mu_\mathfrak B(\beta)+b\binom{v-i}{k}\right),
\end{align*}
from which the claim follows.
Hence, we write $\mu_i=\mu_\mathfrak B(\beta)$, where $i\in\{0,1,\dots,t\}$ and $\beta \in \binom{[v]}{i}$. 

By Lemma~\ref{lem:lambda}, for each $i\in\{0,1,\dots,t\}$ and each $\beta \in \binom{[v]}{i}$, we can write
\begin{align*}
    \lambda_\mathfrak B(\beta)=\sum_{j=0}^i (-1)^j \binom{i}{j}\mu_j.
\end{align*}
Thus $\lambda_\mathfrak B(\beta)$ does not depend on the particular choice of $\beta$. 
Therefore the hypergraph $\mathfrak H_{A}(v,k,a)$ is a $t$-design.  

Finally we determine $\lambda=\lambda_\mathfrak B(\beta)$ where $|\beta|=t$. 
Evaluate the value  
$\displaystyle 
X=\sum_{(\alpha,\beta)\in Y}\det (A[\alpha])
$
in two different ways, where $Y=\{(\alpha,\beta)\in\binom{[v]}{k}\times\binom{[v]}{t}  \; : \;  \beta\subset \alpha\}$. 
On the one hand,
\begin{align*}
X=\sum_{\beta\in\binom{[v]}{t}}\sum_{\substack{\alpha\in\binom{[v]}{k} \\ \beta \subset \alpha }}\det (A[\alpha])&=\sum_{\beta\in\binom{[v]}{t}}\left((a-b)\lambda+b\binom{v-t}{k-t}\right)=\binom{v}{t}\left((a-b)\lambda+b\binom{v-t}{k-t}\right),
\end{align*}
and on the other hand, by Lemma~\ref{lem:coef}(1),
\begin{align*}
    X=\sum_{\alpha\in\binom{[v]}{k}}\sum_{\beta\in\binom{\alpha}{t}}\det (A[\alpha])=\binom{k}{t}\sum_{\alpha\in\binom{[v]}{k}}\det (A[\alpha])=(-1)^k\binom{k}{t}c_A(\emptyset,k). 
\end{align*}  
Therefore, equating these two expressions for $X$ with $a-b\neq0$ yields  
\[
    \lambda=\frac{(-1)^k\binom{k}{t}}{(a-b)\binom{v}{t}}c_A(\emptyset,k)-\frac{b}{a-b}\binom{v-t}{k-t}.\qedhere
\]
\end{proof}
In the remaining subsections of Section~\ref{sec:des}, we apply Theorem~\ref{thm:des} to various families of matrices to obtain infinite families of $t$-designs for $t \in \{2,3\}$.

\subsection{3-designs from conference matrices}
In this subsection, we show how $3$-designs can be obtained from symmetric conference matrices, skew-symmetric conference matrices, and skew Hadamard matrices.  

A \textbf{symmetric (resp.\ skew-symmetric) Seidel matrix} is a $\{0,\pm 1\}$-matrix $S$ with zero diagonal and all off-diagonal entries non-zero such that $S=S^\transpose$ (resp.\ $S=-S^\transpose$).
 A \textbf{conference matrix} is a Seidel matrix $S$ of order $n$ such that $SS^\transpose = (n-1)I$.
Using \cite[Table 4.1]{vLS}, one can obtain Table~\ref{tab:detSeidel34}, which lists $k \times k$ principal minors of a symmetric Seidel matrix. 
\begin{table}[H]
    \centering
        \begin{tabular}{c|cc}
        $k$ &  3 & 4  \\
        \hline
        $D_{S+\varepsilon I}(k)$ & $\{0,-4\varepsilon\}$ & $\{-16,0\}$ \\
        $D_S(k)$ & $\{-2,2\}$ & $\{-3,5\}$  \\
    \end{tabular}
    \caption{The $k\times k$ principal minors of $S+\varepsilon I$ where $S$ is a symmetric Seidel matrix of order $n \geqslant k$ and $\varepsilon = \pm 1$.}
    \label{tab:detSeidel34}
\end{table}
Furthermore, it is a straightforward computation to verify Table~\ref{tab:detskewSeidel345}, which lists $k \times k$ principal minors of a skew-symmetric Seidel matrix. 
\begin{table}[H]
    \centering
        \begin{tabular}{c|ccc}
        $k$ &  3 & 4 & 5 \\
        \hline
        $D_{S + \varepsilon I}(k)$ & $\{-4\varepsilon\}$ & $\{8,16\}$ & $\{-16\varepsilon,-32\varepsilon\}$ \\
        $D_S(k)$ & $\{0\}$ & $\{1,9\}$ & $\{0\}$ \\
    \end{tabular}
    \caption{The $k\times k$ principal minors of $S+\varepsilon I$ where $S$ is a skew-symmetric Seidel matrix of order $n \geqslant k$ and $\varepsilon = \pm 1$.}
    \label{tab:detskewSeidel345}
\end{table}

 Next, we define a \textbf{Paley digraph} of order $q$, which provides infinite families of both symmetric and skew-symmetric conference matrices.
 Let $q$ be an odd prime power.
 The vertex set $V$ is the finite field $\mathbb F_q$ and for $u \ne v$ there is an arc from $u$ to $v$ if $u - v$ is a square in $\mathbb F_q$.
 
 Let $G$ be a Paley digraph of order $q$.
 One can obtain a conference matrix $C$ of order $q+1$ as follows.
 Join a vertex $w$ to $G$ so that there is an arc from $w$ to each vertex of $G$ and, if $q \equiv 1 \pmod 4$, also include an arc from each vertex of $G$ back to $w$.
 Now, set the $(u,v)$ entry of $C$ equal to $-1$ if there is an arc from $u$ to $v$ and $1$ otherwise.
 Set each diagonal entry of $C$ to be $0$.
 Then $C$ is a conference matrix of order $q+1$.
 Note that, if $q \equiv 1 \pmod 4$ then $-1$ is a square in $\mathbb F_q$, which implies that $(u,v)$ is an arc if and only if $(v,u)$ is an arc.
 Thus, $C$ is symmetric if $q \equiv 1 \pmod 4$.
 On the other hand, if $q \equiv 3 \pmod 4$ then $-1$ is not a square in $\mathbb F_q$.
 In this case, $(u,v)$ is an arc if and only if $(v,u)$ is \emph{not} an arc, i.e., the Paley digraph $G$ is a \textit{tournament}.
 Thus, $C$ is skew-symmetric if $q \equiv 3 \pmod 4$.



\begin{example}
\label{ex:sym3}
Let $S$ be a symmetric conference matrix of order $4n+2 \geqslant 6$ and let $\alpha \subset \{1,\dots,4n+2\}$. 
From Table~\ref{tab:detSeidel34}, we have $|D_S(k)|=2$ for $k \in \{3,4\}$.
Using \cite[Section 4]{GS}, one obtains the characteristic polynomials of principal submatrices of $S$ as given in Table~\ref{tab:charSym3}.
\begin{table}[H]
    \centering
    \begin{tabular}{c|c}
        $|\alpha|$ & $\det(xI-S[\overline \alpha])$  \\
        \hline
        $0$ & $(x^2-4n-1)^{2n+1}$  \\
        $1$ & $x(x^2-4n-1)^{2n}$  \\
        $2$ & $(x^2-1)(x^2-4n-1)^{2n-1}$ \\
        $3$ & $(x\pm 2)(x\mp1)^2(x^2-4n-1)^{2n-2}$  \\
    \end{tabular}
    \caption{The characteristic polynomials of principal submatrices of $S$  from Example~\ref{ex:sym3}.}
    \label{tab:charSym3}
\end{table}
It is clear from Table~\ref{tab:charSym3} that $\left |C_S\left (\binom{[4n+2]}{i},4 \right )\right | = 1$ for each $i \in \{0,1,2,3\}$.
Thus, we can apply Theorem~\ref{thm:des}, to deduce that 
\begin{itemize}
    \item $\mathfrak H_{S}(4n+2,4,5)$ is a $3$-$(4n+2,4,3n)$ design;
    \item  $\mathfrak H_{S}(4n+2,4,-3)$ is a $3$-$(4n+2,4,n-1)$ design.
\end{itemize}
\end{example}

Due to the Paley digraph construction above, symmetric conference matrices of order $n$ are known to exist when $n = q+1$ and $q \equiv 1 \pmod 4$ is a prime power. See \cite{BS} for more constructions. 

\begin{example}
\label{ex:skew3}
Let $S$ be a skew-symmetric conference matrix of order $4n \geqslant 8$ and let $\alpha \subset \{1,\dots,4n\}$. 
From Table~\ref{tab:detskewSeidel345}, we have $|D_S(k)|=2$ for $k \in \{4,5\}$.
Using \cite[Section 4]{GS}, one obtains the characteristic polynomials of principal submatrices of $S$ and $S+I$ as given in Table~\ref{tab:charSkew3}.
\begin{table}[H]
    \centering
    \begin{tabular}{c|c}
        $|\alpha|$ & $\det(xI-S[\overline \alpha])$   \\
        \hline
        $0$ & $(x^2+4n-1)^{2n}$  \\
        $1$ & $x(x^2+4n-1)^{2n-1}$ \\
        $2$ & $(x^2+1)(x^2+4n-1)^{2n-2}$ \\
        $3$ & $x(x^2+3)(x^2+4n-1)^{2n-3}$  \\
    \end{tabular}
    \caption{The characteristic polynomials of principal submatrices of $S$ from Example~\ref{ex:skew3}.}
    \label{tab:charSkew3}
\end{table}
It is clear from Table~\ref{tab:charSkew3} that $\left |C_S\left (\binom{[4n]}{i},4 \right )\right | = 1$ and $\left |C_{S+I}\left (\binom{[4n]}{i},5\right )\right | = 1$ for each $j \in \{0,1,2,3\}$.
Thus, we can apply Theorem~\ref{thm:des}, to deduce that 
\begin{itemize}
    \item $\mathfrak H_{S}(4n,4,1)$ is a $3$-$(4n,4,3(n-1))$ design;
    \item $\mathfrak H_{S}(4n,4,9)$ is a $3$-$(4n,4,n)$ design;
    \item $\mathfrak H_{S\pm I}(4n,5,\mp 16)$ is a $3$-$(4n,5,3(n-1)(n-2))$ design;
    \item  $\mathfrak H_{S\pm I}(4n,5,\mp 32)$ is a $3$-$(4n,5,5n(n-1))$ design.
\end{itemize}
\end{example}

Due to the Paley digraph construction given above, skew-symmetric matrices of order $n$ are known to exist when $n=q+1$ where $q\equiv3\pmod{4}$ is a prime power. 
See \cite[Section 9.3.1]{SY} for further details on the existence of skew-symmetric conference matrices.

Now, we briefly comment on how Example~\ref{ex:skew3} can provide an answer to Question~\ref{qsn:gs}.
For $S$ a skew-symmetric conference matrix of order $4n$, the hypergraph $\mathfrak H_{S}(4n,4,9)$ in Example~\ref{ex:skew3}, is a $3$-$(4n,4,n)$ design. 
It remains to check that each subset of $5$ vertices contains either 0 or 2 hyperedges.
Each subset of $5$ vertices corresponds to a $5 \times 5$ principal submatrix and each hyperedge $\alpha$ satisfies $\det(A[\alpha]) = 9$.
It is a straightforward computation to check that, for any skew-symmetric Seidel matrix $T$ of order $5$, the number of subsets $\alpha \in \binom{[5]}{4}$ such that $\det(T[\alpha]) = 9$ is either $0$ or $2$.


It is conjectured in \cite{S78} that skew-symmetric conference matrices of order $4n$ exist for all $n \in \mathbb N$.
The validity of this conjecture would completely answer Question~\ref{qsn:gs}
via the $3$-$(4n,4,n)$ design $\mathfrak H_{S}(4n,4,9)$ in Example~\ref{ex:skew3}.
In the same way, it would also provide a partial answer to the question of Frankl and F\"{u}redi \cite[Problem~1]{FF}.


\subsection{2-designs from equiangular tight frames and Hadamard matrices}
In this subsection, we show how $2$-designs can be obtained from Seidel matrices with precisely two distinct eigenvalues and complex Hadamard matrices over the third roots of unity.
Seidel matrices having precisely two distinct eigenvalues are more commonly known as (real) \textbf{equiangular tight frames}.
See~\cite{etf} for an overview of equiangular tight frames.

\begin{example}\label{ex:ETF}
Let $S$ be a symmetric Seidel matrix of order $v$ with $\det(xI-S) = (x-\theta_1)^{m_1}(x-\theta_2)^{m_2}$ for $\theta_1\neq \theta_2$ and $\min(m_1,m_2) \geqslant 2$. 
Next, one can apply Corollary~\ref{cor:Jacobi22} to obtain the expressions in Table~\ref{tab:charSeidel0} for the characteristic polynomials of principal submatrices of $S$.
\begin{table}[h!]
    \centering
    \begin{tabular}{c|c}
        $|\alpha|$ & $\det(xI-S[\overline \alpha])$  \\
        \hline
        $0$ & $(x-\theta_1)^{m_1}(x-\theta_2)^{m_2}$  \\
        $1$ & $(x-\theta_1)^{m_1-1}(x-\theta_2)^{m_2-1}(x-(\theta_1+\theta_2))$  \\
        $2$ & $(x-\theta_1)^{m_1-2}(x-\theta_2)^{m_2-2}(x-(\theta_1+\theta_2+1))(x-(\theta_1+\theta_2-1))$  \\
    \end{tabular}
    \caption{The characteristic polynomials of principal submatrices of $S$ from Example~\ref{ex:ETF}.}
    \label{tab:charSeidel0}
\end{table}

For $k\in\{3,4\}$ and $\varepsilon \in \{0,\pm 1\}$, we can use Table~\ref{tab:charSeidel0} and Table~\ref{tab:detSeidel34} to deduce that the assumptions of Theorem~\ref{thm:des} are satisfied for $A = S+\varepsilon I$.
Thus, $\mathfrak H_{S+\varepsilon I}(v,k,a)$ is a $2$-$(v,k,\lambda)$ design, where the parameters are given in Table~\ref{tab:ETF}.

\begin{table}[h!]
    \centering
    \begin{tabular}{c|c|c|c}
          $\varepsilon$ & $k$ & $a$ & $\lambda$ \\
        \hline
         $0$ & $3$ & $\pm 2$ & $\frac{\mp 3c_S(\emptyset,3)}{2v(v-1)}+\frac{v-2}{2}$ \\
         $\pm 1$ & $3$ & $0$ & $\frac{-3\varepsilon c_{S+I}(\emptyset,3)}{2v(v-1)}+v-2$ \\
         $\pm 1$ & $3$ & $-4\varepsilon$ & $\frac{3\varepsilon c_{S+I}(\emptyset,3)}{2v(v-1)}$ \\
        $0$ & $4$ & $-3$ & $\frac{-3c_S(\emptyset,4)}{2v(v-1)}+\frac{5}{8}\binom{v-2}{2}$ \\
        $0$ & $4$ & $5$ & $\frac{3c_S(\emptyset,4)}{2v(v-1)}+\frac{3}{8}\binom{v-2}{2}$ \\
        $\pm 1$ & $4$ & $0$ & $\frac{3c_{S\pm I}(\emptyset,4)}{4v(v-1)}+\binom{v-2}{2}$ \\
        $\pm 1$ & $4$ & $-16$ & $\frac{-3c_{S\pm I}(\emptyset,4)}{4v(v-1)}$ \\
    \end{tabular}
    \caption{The $2$-$(v,k,\lambda)$ designs $\mathfrak H_{S+\varepsilon I}(v,k,a)$ of Example~\ref{ex:ETF}.}
    \label{tab:ETF}
\end{table}


\end{example}

\begin{remark}
    When $\min(m_1,m_2)=1$ in Example~\ref{ex:ETF}, we have
    $\mathfrak B_{S}(v,k,a)=\binom{[v]}{k}$, that is, 
    $\mathfrak H_{S}(v,k,a)$ is a trivial design. 
\end{remark}

\begin{remark}
    A \textbf{regular two-graph} is a $2$-$(v,3,\lambda)$ design $(\mathfrak X, \mathfrak B)$ satisfying the property that every subset $\beta \in \binom{\mathfrak X}{4}$ contains an even number of elements of $\mathfrak B$ as subsets.
    Given a regular two-graph one can partition $\binom{\mathfrak X}{4}$ as $\mathfrak C_0 \cup \mathfrak C_1 \cup \mathfrak C_2$, where $\mathfrak C_0$ consists of sets $\beta$ such that $\binom{\beta}{3} \subset \mathfrak B$, $\mathfrak C_2$ consists of sets $\beta$ such that $\binom{\beta}{3} \cap \mathfrak B = \emptyset$, and $\mathfrak C_1 = \binom{\mathfrak X}{4} \backslash (\mathfrak C_0 \cup \mathfrak C_2)$.
    Gillespie~\cite[Proposition 4.1]{G2018} showed that $(\mathfrak X, \mathfrak C_0)$, $(\mathfrak X, \mathfrak C_1)$, and $(\mathfrak X, \mathfrak C_2)$, are each $2$-$(v,4,\lambda)$ designs for differing values of $\lambda$.
    
    It is well-known that regular two-graphs are in correspondence with Seidel matrices having precisely two distinct eigenvalues (see \cite[Theorem 4.1 and Section 7]{Seidel}).
    It is straightforward to verify that this correspondence manifests itself as follows.
    Let $S$ be a Seidel matrix of order $v$ having precisely two distinct eigenvalues.
    Then $\mathfrak H_{S}(v,3,-2)$ is its corresponding regular two-graph.
    Furthermore, we can also recover Gillespie's designs $(\mathfrak X, \mathfrak C_0)$, $(\mathfrak X, \mathfrak C_1)$, and $(\mathfrak X, \mathfrak C_2)$ in our framework. 
    Indeed, one observes that
    \begin{align*}
        \mathfrak C_0 &= \left \{ \beta \in \binom{\mathfrak X}{4} \; : \; \det(xI - S[\beta]) = (x+3)(x-1)^3 \right \}; \\
        \mathfrak C_1 &= \left \{ \beta \in \binom{\mathfrak X}{4} \; : \; \det(xI - S[\beta]) = (x^2-5)(x^2-1) \right \}; \\
        \mathfrak C_2 &= \left \{ \beta \in \binom{\mathfrak X}{4} \; : \; \det(xI - S[\beta]) = (x-3)(x+1)^3 \right \}.
    \end{align*}
    It then follows that $(\mathfrak X, \mathfrak C_0) = \mathfrak H_{S+I}(v,4,-16)$, $(\mathfrak X, \mathfrak C_1)= \mathfrak H_{S}(v,4,5)$, and $(\mathfrak X, \mathfrak C_2)= \mathfrak H_{S-I}(v,4,-16)$. 
\end{remark}

We refer again to \cite{etf} for the existence of (real) equiangular tight frames.
The next example is a special case of Example~\ref{ex:ETF}.
A complex $v \times v$ matrix $H$ is called \textbf{Hadamard} if each of its entries has absolute value $1$ and if $H H^* = vI$, where $H^*$ denotes the Hermitian transpose of $H$.
A real, symmetric Hadamard matrix with a constant diagonal is known as a \textbf{graphical Hadamard} matrix.

\begin{example}\label{ex:graphHadamard}
Let $H$ be a graphical Hadamard matrix of order $n \geqslant 4$.
Then $H=S+\varepsilon I$ for some Seidel matrix $S$ with two distinct eigenvalues and $\varepsilon \in \{-1,1\}$.
Furthermore, $\det(xI-H) = (x-\sqrt{n})^{m}(x+\sqrt{n})^{n-m}$, where $m = \sqrt{n}(\sqrt{n}+\varepsilon)/2$.
Here, we can compute the expressions $c_H(\emptyset,3) = \varepsilon n^2(n-1)/3$ and $c_H(\emptyset,4) = -n^3(n-1)/12$.
Now we can simplify the expressions in Table~\ref{tab:ETF} to find
\begin{itemize}
    \item $\mathfrak H_{H}(n,3,0)$ is a $2$-$(n,3,(3n-8)/4)$ design;
    \item $\mathfrak H_{H}(n,3,-4\varepsilon)$ is a $2$-$(n,3,n/4)$ design;
    \item $\mathfrak H_{H}(n,4,0)$ is a $2$-$(n,4,(7n^2-24n+16)/16)$ design;
    \item  $\mathfrak H_{H}(n,4,\pm 16)$ is a $2$-$(n,4,n^2/16)$ design.
\end{itemize}

Note that, in order for $\lambda$ to be an integer, $n$ must be a multiple of $4$.
Various infinite families of graphical Hadamard matrices are known to exist.
We refer the reader to \cite{BrouwerBlokhuis60} for details.
\end{example}

Let $\omega$ be a primitive third root of unity.
Our next example produces $2$-designs from Hermitian Hadamard matrices over the third roots of unity.

\begin{example}
\label{ex:herm}
Let $H$ be a Hermitian Hadamard matrix of order $n^2 \geqslant 9$ with entries from $\{1,\omega,\omega^2\}$. 
It is a straightforward computation to verify Table~\ref{tab:hermMinors}, which lists all $k \times k$ principal minors of $H$ where $k \in \{3,4\}$.
\begin{table}[h!]
\label{tab:hermMinors}
\centering
        \begin{tabular}{c|c|c}
        $k$ &  3 & 4  \\
        \hline
        $D_{H}(k)$ & $\{-3,0\}$ & $\{-9,0\}$  \\
    \end{tabular}
    \caption{The $k\times k$ principal minors of the matrix $H$ from Example~\ref{ex:herm}.}
\end{table}

Next, one can apply Corollary~\ref{cor:Jacobi22} to obtain the expressions in Table~\ref{tab:charthird0} for the characteristic polynomials of principal submatrices of $H$.
\begin{table}[h!]
    \centering
    \begin{tabular}{c|c}
        $|\alpha|$ & $\det(xI-H[\overline \alpha])$  \\
        \hline
        $0$ & $(x-n)^{\frac{n^2+n}{2}}(x+n)^{\frac{n^2-n}{2}}$  \\
        $1$ & $(x+1)(x-n)^{\frac{n^2+n}{2}-1}(x+n)^{\frac{n^2-n}{2}-1}$  \\
        $2$ & $x(x+2)(x-n)^{\frac{n^2+n}{2}-2}(x+n)^{\frac{n^2-n}{2}-2}$  \\
    \end{tabular}
    \caption{The characteristic polynomials of principal submatrices of $H$ from Example~\ref{ex:herm}.}
    \label{tab:charthird0}
\end{table}

Thus, we can apply Theorem~\ref{thm:des}, to deduce that 
\begin{itemize}
    \item $\mathfrak H_{H}(n^2,3,-3)$ is a $2$-$(n^2,3,\frac{2 n^2}{3})$ design;
    \item $\mathfrak H_{H}(n^2,3,0)$ is a $2$-$(n^2,3,\frac{n^2-6}{3})$ design;
    \item $\mathfrak H_{H}(n^2,4,-9)$ is a $2$-$(n^2,4,\frac{n^4}{9})$ design;
    \item  $\mathfrak H_{H}(n^2,4,0)$ is a $2$-$(n^2,4,\frac{7 n^4}{18}-\frac{5 n^2}{2}+3)$ design.
\end{itemize}
 \end{example}

A \textbf{Butson Hadamard matrix} of order $n$ and phase $q$, denoted by $\text{BH}(n,q)$, is a complex Hadamard matrix of order $n$ whose entries are $q$-th roots of unity. 
It is easy to see that a BH$(n,3)$ exists only if $n\equiv0\pmod{3}$. 
Moreover, BH$(n,3)$ matrices are known to exist for $n\in \{2^a3^b(21)^c \; :\; 0\leqslant a\leqslant b,1\leqslant c\}$ \cite{AOS,butson}. 
By \cite[Theorem~14]{KS90}, if there exists BH$(n,3)$, then there exists a Hermitian BH$(n^2,3)$. 




A \textbf{doubly regular tournament} is a tournament whose adjacency matrix $A$ satisfies $AA^\transpose = t J + (t+1)I$ for some $t$. 
The following example produces $2$-designs from doubly regular tournaments. 
\begin{example}
\label{ex:doubly}
Let $A$ be the adjacency matrix of a doubly regular tournament of order $4n+3 \geqslant 7$.  
It is a straightforward computation to verify Table~\ref{tab:det01skewSym34}, which lists all $k \times k$ principal minors of $A$ and $A+I$ for $k \in \{3,4\}$.
\begin{table}[H]
\label{tab:det01skewSym34}
    \centering
        \begin{tabular}{c|c|c}
        $k$ &  3 & 4  \\
        \hline
        $D_A(k)$ & $\{0,1\}$ & $\{-1,0\}$  \\
        $D_{A+I}(k)$ & $\{1,2\}$ & $\{1,2\}$  \\
    \end{tabular}
    \caption{The $k \times k$ principal minors of $A$ and $A+I$ from Example~\ref{ex:doubly}.}
    
\end{table}
By \cite[Theorem 1]{B}, we can obtain the formulas of Table~\ref{tab:chardoubly0}.
\begin{table}[H]
    \centering
    \begin{tabular}{c|c}
        $|\alpha|$ & $\det(xI-A[\overline \alpha])$  \\
        \hline
        $0$ & $(x-2n-1)(x^2+x+n+1)^{2n+1}$  \\
        $1$ & $(x^2+x+n+1)^{2n}(x^2-2nx-n)$  \\
        $2$ & $(x^2+x+n+1)^{2n-1}(x^3-(2n-1)x^2-(2n-1)x-n)$  \\
    \end{tabular}
    \caption{The characteristic polynomials of principal submatrices of $A$ from Example~\ref{ex:doubly}.}
    \label{tab:chardoubly0}
\end{table}



Now we can apply Theorem~\ref{thm:des}, to each row $(M,k,a,\lambda)$ of Table~\ref{tab:doubly}, to deduce that the hypergraph $\mathfrak H_{M}(4n+3,k,a)$ is a $2$-$(4n+3,k,\lambda)$ design.
\begin{table}[h!]
    \centering
    \begin{tabular}{c|c|c|c}
        $M$ &  $k$ & $a$ & $\lambda$ \\
        \hline
        $A$ & $3$ & $0$ & $3n$ \\
        $A$ & $3$ & $1$ & $n+1$ \\
        $A$ & $4$ & $-1$ & $3n(n+1)$ \\
        $A$ & $4$ & $0$ & $n(5n-1)$ \\
        $A+I$ & $4$ & $1$ & $3n(n-1)$ \\
        $A+I$ & $4$ & $2$ & $5n(n+1)$ \\
    \end{tabular}
    \caption{The $2$-$(4n+3,k,\lambda)$ designs $\mathfrak H_{M}(4n+3,k,a)$ of Example~\ref{ex:doubly}.}
    \label{tab:doubly}
\end{table}

It turns out that the hypergraphs $\mathfrak H_{A\pm I}(4n+3,3,\pm 1)$ are $2$-$(4n+3,3,3n)$ designs and $\mathfrak H_{A\pm I}(4n+3,3,1\pm 1)$ are $2$-$(4n+3,3,n+1)$ designs.
\end{example}

By \cite[Theorem 2]{RB72}, the existence of doubly regular tournaments of order $n$ is equivalent to that of so-called \textit{skew Hadamard matrices} of order $n+1$.
Note that $H$ is a skew Hadamard matrix if and only if $H-I$ is a skew-conference matrix.   
We refer again the reader to \cite[Section 9.3.1]{SY} for further details on the existence of skew Hadamard matrices.

\section{Partially balanced incomplete block designs
}\label{sec:pbibd}
In this section, we obtain partially balanced incomplete block designs from a matrix $A$ in the Bose-Mesner algebra of a symmetric association scheme.
The parameters of the resulting PBIBDs can be determined from the spectral data of the matrix $A$. 

\subsection{Symmetric association schemes}
\label{sec:boseMes}
Let $d$ be a positive integer. 
Let $\mathfrak X$ be a finite set and $R_i$ ($i\in\{0,1,\ldots,d\}$) be a nonempty subset of $\mathfrak X\times \mathfrak X$. 
The \textbf{adjacency matrix} $A_i$ of the graph with vertex set $\mathfrak X$ and edge set $R_i$ is a $\{0,1\}$-matrix indexed by $\mathfrak X$ such that $(A_i)_{xy}=1$ if $(x,y)\in R_i$ and $(A_i)_{xy}=0$ otherwise. 
A \textbf{{\rm(}symmetric{\rm)} association scheme} with $d$ classes is a pair $\mathcal{X}=(\mathfrak X,\{R_i\}_{i=0}^d)$ satisfying the following:
\begin{enumerate}
\item $A_0=I_{|\mathfrak X|}$.
\item $\sum_{i=0}^d A_i = J_{|\mathfrak X|}$.
\item $A_i^\transpose =A_i$ for any $i\in\{1,\ldots,d\}$.
\item For any $i$ and $j$, the product $A_i A_j$ belongs to the linear span of the matrices $A_0,\dots,A_d$.
\end{enumerate}
The algebra generated by $A_0,A_1,\ldots,A_d$ over $\mathbb{R}$ is said to be the \textbf{Bose-Mesner algebra} of a symmetric association scheme.
An association scheme $\mathcal{X}=(\mathfrak X,\{R_i\}_{i=0}^d)$ can be defined indirectly in terms of the adjacency matrices $A_0,\dots,A_d$ (see Section~\ref{sec:bgw}).

A \textbf{partially balanced incomplete block design} (PBIBD) on the association scheme $(\mathfrak X,\{R_i\}_{i=0}^d)$ is a $1$-design $(\mathfrak X,\mathfrak B)$ with the property that 
for any $i$, there exists $\lambda_i$ such that for any pair $(x,y)\in R_i$, $\{x,y\}$ is contained in exactly $\lambda_i$ hyperedges of $\mathfrak B$ \cite{gdd1,LZ,gdd2,Street}.  
We denote it as $\text{PBIBD}(v,k;\lambda_1,\ldots,\lambda_d)$. 
Note that if $\lambda_1=\cdots=\lambda_d$ then the PBIBD is a $2$-design. 


The following is an analogue of Theorem~\ref{thm:des} to the case of association schemes. 
For a binary relation $R_i$, define  $\tilde{R}_i=\{\{x,y\}  \; : \;  (x,y)\in R_i\}$.

\begin{theorem}\label{thm:pbibd}
Let $\mathcal{X}=([v],\{R_i\}_{i=0}^d)$ be a symmetric association scheme. 
Let $k$ be a positive integer, and $A$ be a $v\times v$ complex diagonalisable matrix such that  
\begin{enumerate}
    \item $D_A(k)=\{a,b\},a\neq b$, and
    \item $C_A \left (\binom{[v]}{i},k \right) = \{c_i\}$ for each $i \in \{0,1\}$ and $C_A(\tilde{R}_j,k) = \{c_2(j)\}$ for each $j \in \{1,\dots,d\}$.
\end{enumerate}
Then the hypergraph $\mathfrak H_{A}(v,k,a)$ 
is a $\operatorname{PBIBD}(v,k;\lambda_1,\ldots,\lambda_d)$ on $\mathcal{X}$ with 
\[
\lambda_i=\frac{(-1)^k}{a-b}\left(c_0-2c_1+c_2(i)\right)-\frac{b}{a-b}\binom{v-2}{k-2}.
\]
\end{theorem}
\begin{proof}
Let $\mathfrak B = \mathfrak B_A(v,k,a)$.
By Corollary~\ref{cor:mu} 
together with assumption (1), we can obtain, for each $i\in\{0,1\}$, 
\begin{align*}
    c_i &= c_A(\beta,k)=(-1)^{k}\left((a-b)\mu_\mathfrak B(\beta)+b\binom{v-i}{k}\right)\text{ for }\beta\in\binom{[v]}{i}.
\end{align*}
It follows that, for each $i\in\{0,1\}$, there exists $\mu_i$ such that $\mu_\mathfrak B(\beta) = \mu_i$ for each $\beta \in \binom{[v]}{i}$.
Next, for any $\beta \in \binom{[v]}{2}$, there exists $j \in \{1,\dots,d\}$ such that $\beta \in \tilde{R}_j$.
Furthermore, we claim that, for each $j \in \{1,\dots,d\}$, there exists $\mu_2(j)$ such that $\mu_\mathfrak B(\beta) =\mu_2(j)$ for each $\beta \in \tilde{R}_j$.
Indeed, by Corollary~\ref{cor:mu} 
together with assumption (1), we can obtain, for each $j \in \{1,\dots,d\}$,
\begin{align*}
    c_2(j) &= c_A(\beta,k)=(-1)^{k}\left((a-b)\mu_\mathfrak B(\beta)+b\binom{v-2}{k}\right)\text{ for }\beta\in\tilde{R}_j,
\end{align*}
from which our claim follows.
Therefore, for each $i\in\{0,1\}$, we can write $\mu_i=\mu_\mathfrak B(\beta)$ for some $\beta \in \binom{[v]}{i}$ and $\mu_{2}(j)=\mu_\mathfrak B(\beta)$ for some $\beta \in \tilde{R}_j$.

By Lemma~\ref{lem:lambda},
$$
\lambda_\mathfrak B(\beta)=\sum_{\gamma\subset \beta}(-1)^{|\gamma|}\mu_\mathfrak B(\gamma)=\begin{cases}\mu_0 & \text{ if }|\beta|=0,\\
\mu_0-\mu_1 & \text{ if }|\beta|=1,\\
\mu_0-2\mu_1+\mu_{2}(j) & \text{ if }|\beta|=2 \text{ 
and }\beta\in \tilde{R}_j.
\end{cases} 
$$
Thus, $\lambda_\mathfrak B(\beta)$ does not depend on $\beta$ for $\beta \in \binom{[v]}{i}$ when $i\in\{0,1\}$ and depends only on $j$  if $|\beta|=2$ and $\beta\in \tilde{R}_j$.  
Hence, the hypergraph $\mathfrak H_{A}(v,k,a)$ 
is a $\operatorname{PBIBD}(v,k;\lambda_1,\ldots,\lambda_d)$ on $\mathcal{X}$.

Finally, we establish the expressions for the parameters $\lambda_1,\ldots,\lambda_d$. 
Since 
\begin{align*}
\mu_i&=\frac{1}{a-b}\left((-1)^{k} c_{j}-b\binom{v-i}{k}\right)\quad \text{ for }i\in\{0,1\},\\  
\mu_{2}(j)&=\frac{1}{a-b}\left((-1)^k c_{2}(j)-b\binom{v-2}{k}\right)\quad \text{ for }j\in\{1,\ldots,d\},   
\end{align*}
the parameters $\lambda_j$ can be determined as 
\begin{align*}
\lambda_j&=\mu_0-2\mu_1+\mu_{2}(j)\\
    &=\frac{(-1)^k}{a-b}\left(c_0-2c_1+c_{2}(j)\right)-\frac{b}{a-b}\left(\binom{v}{k}-2\binom{v-1}{k}+\binom{v-2}{k}\right)\\
    &=\frac{(-1)^k}{a-b}\left(c_0-2c_1+c_{2}(j)\right)-\frac{b}{a-b}\binom{v-2}{k-2}.\qedhere
\end{align*}
\end{proof}

The following lemma shows that a matrix $A$ taken from the Bose-Mesner algebra in an association scheme satisfies the assumption of Theorem~\ref{thm:pbibd}.  

\begin{lemma}\label{lem:as}
Let $\mathcal{X}=([v],\{R_i\}_{i=0}^d)$ be a symmetric association scheme and let $A$ be an element of its Bose-Mesner algebra.
Then
\begin{enumerate}
    \item $\left | \left \{ \det(xI-A[\overline{\beta}]) \; : \; \beta\in \binom{[v]}{1} \right \} \right | = 1$;
\item for each $i \in [d]$, we have $ \left | \left \{ \det(xI-A[\overline{\beta}]) \; : \; \beta\in \tilde{R}_i \right \} \right | = 1$.
\end{enumerate}
\end{lemma}
\begin{proof}
Suppose $\det(xI-A) = \prod_{i=1}^s(x-\theta_i)^{m_i}$.
Since $A^j=\sum_{k=0}^d a_{j}(k)A_k$ for some coefficients $a_{j}(k)$,  by Lemma~\ref{lem:Jacobi2}, we have 
\begin{align*}
    \det(xI-A[\overline{\beta}]) &=\prod_{i=1}^s(x-\theta_i)^{m_i-|\beta|} \cdot \det \left (\sum_{i=1}^s \left(\sum_{j=0}^i (-1 )^j e_j(\theta_1,\dots,\theta_s)x^{i-j}\right)A^{s-i}[\beta] 
 \right )\\
 &=\prod_{i=1}^s(x-\theta_i)^{m_i-|\beta|} \cdot \det \left (\sum_{i=1}^s \left(\sum_{j=0}^i (-1 )^j e_j(\theta_1,\dots,\theta_s)x^{i-j}\right)\left(\sum_{k=0}^d a_{s-i}(k)A_k\right)[\beta] 
 \right )\\
 &=\prod_{i=1}^s(x-\theta_i)^{m_i-|\beta|} \cdot \det \left(\sum_{k=0}^d\left (\sum_{i=1}^s a_{s-i}(k)\sum_{j=0}^i (-1 )^j e_j(\theta_1,\dots,\theta_s)x^{i-j}\right) A_k[\beta] 
 \right ).
\end{align*}
Denote by $\delta_{k,\ell}$ Kronecker's delta.
For a subset $\beta$ with $|\beta|=1$, we have $A_k[\beta]=\delta_{k,0}$. 
For subsets $\beta$ such that $|\beta|=2$,
we have $\beta\in\tilde{R}_\ell$ for some $\ell \in [d]$.
Furthermore
$A_k[\beta]=\delta_{k,\ell}(J_2-I_2)$, whence the conclusions follow. 
\end{proof}

Now we can combine Theorem~\ref{thm:pbibd} together with Lemma~\ref{lem:as} to obtain the following corollary.

\begin{corollary}
\label{cor:BoseMesnerPBIBD}
    Let $\mathcal{X}=([v],\{R_i\}_{i=0}^d)$ be a symmetric association scheme and let $A$ be an element of its Bose-Mesner algebra.
    Suppose that $|D_A(k)| = 2$ for some positive integer $k$.
    Then, for $a \in D_A(k)$, the hypergraph $\mathfrak H_{A}(v,k,a)$ is a PBIBD on $\mathcal{X}$.
\end{corollary}

We demonstrate the utility of Corollary~\ref{cor:BoseMesnerPBIBD} with the following two examples.

\begin{example} 
\label{ex:symassPBIBD1}
Let $\mathcal{X}=([v],\{R_i\}_{i=0}^d)$ be a symmetric association scheme with adjacency matrices $A_0,\dots,A_d$.
    Suppose that $k=3$ and $A$ is a nonzero, binary combination of $A_0,\dots,A_d$, i.e., $A$ is a $\{0,1\}$-matrix is the Bose-Mesner algebra of $\mathcal X$.
    Since each $A_i$ is symmetric, so too is $A$ and thus, $A$ is the adjacency matrix of a graph.
Using Table~\ref{tab:det01skewSym34} one finds that $|D_A(3)|=2$.
Hence, for $a \in D_A(3)$, by Corollary~\ref{cor:BoseMesnerPBIBD}, the hypergraph $\mathfrak H_{A}(v,3,a)$ is a PBIBD on $\mathcal{X}$. 
\end{example}

\begin{example} 
\label{ex:symassPBIBD2}
Let $\mathcal{X}=([v],\{R_i\}_{i=0}^d)$ be a symmetric association scheme with adjacency matrices $A_0,\dots,A_d$.
    Suppose that $k\in\{3,4\}$ and $S=\sum_{i=0}^d a_i A_i$ where $a_0\in \{-1,0,1\}$ and $a_i\in\{-1,1\}$ for each $i\in\{1,\ldots,d\}$.
    Note that if $a_0=0$, $S$ is a symmetric Seidel matrix in the Bose-Mesner algebra of $\mathcal X$.
Using Table~\ref{tab:detSeidel34}, one finds that $|D_A(k)|=2$.
Hence, for $a \in D_S(k)$, by Theorem~\ref{cor:BoseMesnerPBIBD}, the hypergraph $\mathfrak H_{S}(v,k,a)$ is a PBIBD on $\mathcal{X}$.  
\end{example}

We remark that the family of symmetric association schemes includes strongly regular graphs and distance-regular graphs, which are a subject of intense research interest in their own right. 
See \cite{SRG} for strongly regular graphs and \cite{BCN, DRG} for distance regular graphs.

\subsection{Skew-Bush type Hadamard matrices and group divisible designs}\label{sec:gdd}
Next, we give an example of PBIBD on a $2$-class association scheme from a skew-Bush type Hadamard matrix. 

A Hadamard matrix of order $n^2$ is said to be \textbf{skew-Bush type} if $H$ is of block form $H=(H_{ij})_{i,j=1}^n$,  where each $H_{ij}$ is an $n\times n$ matrix, such that $H_{ii}=J$ for any $i$ and $H_{ij}J=JH_{ij}=O$, $H_{ij}^\transpose=-H_{ji}$ for any 
distinct $i,j$. 
Note that $H+H^\transpose=2I_n\otimes J_n$, $HJ=JH=nJ$ and $H (I_n\otimes J_n)=(I_n\otimes J_n)H=nI_n\otimes J_n$. 
Define the $\{0,1\}$-matrix $A$ by $A=(J-H)/2$.
Note that $A$ is the adjacency matrix of a tournament, see \cite{KS17}.
Moreover, $A$ satisfies the equations
 $AJ=JA=\frac{n^2-n}{2}J$, $A(I_n\otimes J_n)=(I_n\otimes J_n)A=\frac{n}{2}(J-I_n\otimes J_n)$, and  
\begin{align}
    A^2=\frac{1}{4}(J^2-JH-HJ+H^2)
    &=\frac{1}{4}((n^2-2n) J-n^2I+2nI_n\otimes J_n).\label{eq:sb}
\end{align}
It follows that the eigenvalues of $A$ are 
$\frac{n(n-1)}{2},-\frac{n}{2},\frac{n\sqrt{-1}}{2},-\frac{n\sqrt{-1}}{2}$ with multiplicities $1,n-1,\frac{n^2-n}{2},\frac{n^2-n}{2}$, respectively. 
Let
\begin{align*}
    R_0 &=\{(x,x)  \; : \;  x\in [n^2] \}; \\
    R_1 &=\{(x,y) \in [n^2] \times [n^2]  \; : \;  \lfloor (x-1)/n\rfloor=\lfloor (y-1)/n\rfloor,x\neq y \}; \text{ and } \\
    R_2 &=\{(x,y)  \; : \;  \lfloor (x-1)/n\rfloor\neq \lfloor (y-1)/n\rfloor \}.
\end{align*}

The pair $\mathcal X_1 := ([n^2],\{R_0,R_1,R_2\})$ is a symmetric association scheme. 
PBIBDs on this scheme are known as \emph{group divisible designs}~\cite{gdd1,gdd2}.
As before, define $\tilde{R}_i=\{\{x,y\}  \; : \;  (x,y)\in R_i\}$ for each $i\in\{0,1,2\}$. 

In this case, the matrix $A$ does not belong to the Bose-Mesner algebra of $\mathcal X_1$.
Hence, we cannot merely apply Corollary~\ref{cor:BoseMesnerPBIBD}.
To get around this, we check that assumption (2) of Theorem~\ref{thm:pbibd} holds in the next lemma.

\begin{lemma}\label{lem:skewgdd}
Let $H$ be a skew-Bush type Hadamard matrix of order $n^2$ and $A=(J-H)/2$. 
Then
\begin{enumerate}
    \item $\left | \left \{ \det(xI-A[\overline{\alpha}]) \; : \; \alpha\in \binom{[n^2]}{1} \right \} \right | = 1$.  
\item for each $i \in \{1,2\}$, we have $ \left | \left \{ \det(xI-A[\overline{\alpha}]) \; : \; \alpha\in \tilde{R}_i \right \} \right | = 1$.
\end{enumerate}
\end{lemma}
\begin{proof}
As shown above, the characteristic polynomial of $A$ is $\det (xI-A)=\prod_{i=1}^4(x-\theta_i)^{m_i}$ where $\theta_1=\frac{n(n-1)}{2},\theta_2=-\frac{n}{2},\theta_3=\frac{n\sqrt{-1}}{2},\theta_4=-\frac{n\sqrt{-1}}{2}$ and $m_1=1,m_2=n-1,m_3=\frac{n^2-n}{2},m_4=\frac{n^2-n}{2}$.
By Lemma~\ref{lem:Jacobi2}, 
we have
\begin{align*}
    \frac{\det(xI-A[\overline{\alpha}])}{\det(xI-A)}
    &= \frac{\det \left ( A^3[\alpha]+(x-e_1)A^2[\alpha]+(x^2-e_1x+e_2)A[\alpha]+(x^3-e_1x^2+e_2x-e_3)I \right )}{\prod_{i=1}^4(x-\theta_i)^{|\alpha|-m_i}}, 
\end{align*}
where $e_j=e_j(\theta_1,\ldots,\theta_4)$ for $j\in\{1,2,3\}$. 
Here, using \eqref{eq:sb} and the fact that 
\begin{align*}
A^3=\frac{1}{4}\left(\frac{n^2(n^2-3 n+4)}{2}J-n^2 A-n^2I_n\otimes J_n\right),
\end{align*}
we find that 
\begin{align*}
    \ A^3[\alpha]&+(x-e_1)A^2[\alpha]+(x^2-e_1x+e_2)A[\alpha]+(x^3-e_1x^2+e_2x-e_3)I\\ 
   =&\frac{n (2 (n-2) x+n^2)}{8} J+\frac{n(2 x-n^2+n)}{4} \left(I_n\otimes J_n\right)[\alpha]+\left( x^2+\frac{2n-n^2}{2}x+\frac{n^2-n^3}{4}  \right)A[\alpha]\\
   &+\left(x^3+(n-\frac{n^2}{2})x^2+\frac{n^2-n^3}{4}x\right)I. 
\end{align*}
Since 
\begin{align*}
    A[\alpha]=\begin{cases}
    0 & \text{ if } |\alpha|=1,\\
    O_2  & \text{ if } |\alpha|=2\text{ and }\alpha\in \tilde{R}_1,\\
    \left[\begin{smallmatrix} 0 & 1 \\ 0 & 0 \end{smallmatrix}\right]\text{ or } \left[\begin{smallmatrix} 0 & 0 \\ 1 & 0 \end{smallmatrix}\right] & \text{ if }|\alpha|=2\text{ and } \alpha\in \tilde{R}_2,
    \end{cases}
    \qquad
    (I_n\otimes J_n)[\alpha]=\begin{cases}
    1 & \text{ if } |\alpha|=1,\\
    J_2  & \text{ if } |\alpha|=2\text{ and }\alpha\in \tilde{R}_1,\\
    I_2 & \text{ if } |\alpha|=2\text{ and }\alpha\in \tilde{R}_2,
    \end{cases}
\end{align*}
the conclusion follows. 
\end{proof}

 Now we are ready to apply Theorem~\ref{thm:pbibd} to $\mathcal X_1$ and $A$.

\begin{corollary}\label{cor:gdd}
Let $H$ be a skew-Bush type Hadamard matrix of order $n^2$.
Suppose $A=(J-H)/2$, $k\in\{3,4\}$, and $a\in D_A(k)$. 
Then $\mathfrak H_A(n^2,k,a)$ is a $\operatorname{PBIBD}(n^2,k;\lambda_1,\lambda_2)$ on $\mathcal X_1$, where
\begin{align*}
(\lambda_1,\lambda_2)=\begin{cases}
\left (n^2-2,\frac{(3n-4)(n+2)}{4} \right) & \text{ if }(k,a)=(3,0),\\
\left (0,\frac{n(n-2)}{4} \right) & \text{ if }(k,a)=(3,1),\\
\left (\frac{n^2}{16},\frac{n^2(3n^2-10n+12)}{16}\right ) & \text{ if }(k,a)=(4,-1),\\
\left (\frac{(n-2)(n+2)(7n^2-12)}{16},\frac{(n-2)(55n^3+20n^2-12n-24)}{16}\right ) & \text{ if }(k,a)=(4,0).
\end{cases}
\end{align*}
\end{corollary}
\begin{proof}
As shown in Table~\ref{tab:det01skewSym34}, we have $|D_A(k)| = 2$.
By Lemma~\ref{lem:skewgdd}, assumption (2) in Theorem~\ref{thm:pbibd} is also satisfied.
Hence, we can apply Theorem~\ref{thm:pbibd} to deduce that the hypergraph $\mathfrak H_A(n^2,k,a)$ is a PBIBD on $\mathcal X_1$. 
By expanding $\det(xI -A[\overline{\alpha}])$ in the proof of Lemma~\ref{lem:skewgdd} and substituting their coefficients to $\lambda_i$ in Theorem~\ref{thm:pbibd}, 
the parameters $(\lambda_1,\lambda_2)$ can be calculated as required. 
\end{proof}

\subsection{Extending to three distinct $k \times k$ principal minors}

The examples in the subsequent subsections do not satisfy the assumptions of Theorem~\ref{thm:pbibd}.
To make our method more flexible, we extend Theorem~\ref{thm:pbibd}, by allowing the case of a complex matrix $A$ satisfying $|D_A(k)|=3$ for some $k$.
As above, for a binary relation $R_i$, we define  $\tilde{R}_i=\{\{x,y\}  \; : \;  (x,y)\in R_i\}$.

\begin{theorem}\label{thm:pbibd2}
Let $\mathcal{X}=([v],\{R_i\}_{i=0}^d)$ be a symmetric association scheme. 
Let $k$ be a positive integer, and $A$ be a $v\times v$ complex diagonalisable matrix such that  
\begin{enumerate}
    \item $D_A(k)=\{a,b,c\}$, where $|D_A(k)| = 3$,
    \item $C_A \left (\binom{[v]}{i},k \right) = \{c_i\}$ for each $i \in \{0,1\}$ and $C_A(\tilde{R}_j,k) = \{c_2(j)\}$ for each $j \in [d]$, and 
    \item 
    the hypergraph $\mathfrak H_{A}(v,k,c)$  is a $\operatorname{PBIBD}(v,k;\eta_2(1),\dots,\eta_2(d))$ on $\mathcal{X}$.
\end{enumerate}
Then the hypergraph $\mathfrak H_{A}(v,k,a)$ 
is a $\operatorname{PBIBD}(v,k;\lambda_1,\dots,\lambda_d)$ on $\mathcal{X}$ where, for each $j \in [d]$, we have
\[
\lambda_j=\frac{(-1)^k}{a-b}\left(c_0-2c_1+c_{2}(j)\right)-\frac{b}{a-b}\binom{v-2}{k-2}-\frac{c-b}{a-b}\eta_2(j).
\]
\end{theorem}
\begin{proof}
Let $\mathfrak B = \mathfrak B_A(v,k,a)$, $\eta_{0} = |\mathfrak B_A(v,k,c)|$, and $\eta_1 = \left |\left \{ \alpha \in \mathfrak B_A(v,k,c) \; : \; 1 \in \alpha \right \} \right |$.
First we claim that, for each $i\in\{0,1\}$, there exists $\mu_i$ such that $\mu_\mathfrak B(\beta) = \mu_i$ for each $\beta \in \binom{[v]}{i}$. 
For $x \in [v]$, denote by $\nu_1$ the number of hyperedges of $\mathfrak H_{A}(v,k,c)$ that do not contain $x$.
For $\{x,y\}\in \tilde{R}_i$, denote by $\nu_2(i)$ the number of hyperedges of $\mathfrak H_{A}(v,k,c)$ that do not contain $\{x,y\}$.
Using a straightforward inclusion-exclusion argument, we can write $\nu_1 = v - \eta_1$ and $\nu_2(i) = v - 2\eta_1+\eta_2(i)$ for each $i \in [d]$.
Also, define $\nu_0 = \eta_0$.
By Corollary~\ref{cor:mu}
together with assumptions (1) and (3), for each $i\in\{0,1\}$, we obtain
\begin{align*}
    c_i &= c_A(\beta,k)=(-1)^{k}\left(a\mu_\mathfrak B(\beta)+b\left(\binom{v-i}{k}-\nu_i-\mu_\mathfrak B(\beta)\right)+c\nu_i\right)\nonumber\\
    &= (-1)^{k}\left((a-b)\mu_\mathfrak B(\beta)+b\binom{v-i}{k}+(c-b)\nu_i\right)\text{ for }\beta\in\binom{[v]}{i}.
\end{align*}
from which our claim follows.
Next, for any $\beta \in \binom{[v]}{2}$, there exists $j \in \{1,\dots,d\}$ such that $\beta \in \tilde{R}_j$.
Furthermore, we claim that, for each $j \in \{1,\dots,d\}$, there exists $\mu_2(j)$ such that $\mu_\mathfrak B(\beta) =\mu_2(j)$ for each $\beta \in \tilde{R}_j$.
Again, by Corollary~\ref{cor:mu}
together with assumption (1) and (3), for each $j \in \{1,\dots,d\}$, we obtain
\begin{align*}
    c_2(j) &=  c_A(\beta,k)=(-1)^{k}\left(a\mu_\mathfrak B(\beta)+b\left(\binom{v-2}{k}-\nu_2(j)-\mu_\mathfrak B(\beta)\right)+c\nu_2(j)\right)\nonumber    \\
    &= (-1)^{k}\left((a-b)\mu_\mathfrak B(\beta)+b\binom{v-2}{k}+(c-b)\nu_2(j)\right)\text{ for }\beta\in\tilde{R}_j,
\end{align*}
from which the claim follows.
Hence, for each $i\in\{0,1\}$, we can write $\mu_i=\mu_\mathfrak B(\beta)$ for some $\beta \in \binom{[v]}{i}$ and $\mu_{2}(j)=\mu_\mathfrak B(\beta)$ for some $\beta \in \tilde{R}_j$.

By Lemma~\ref{lem:lambda},
$$
\lambda_\mathfrak B(\beta)=\sum_{\gamma\subset \beta}(-1)^{|\gamma|}\mu_\mathfrak B(\gamma)=\begin{cases}\mu_0 & \text{ if }|\beta|=0,\\
\mu_0-\mu_1 & \text{ if }|\beta|=1,\\
\mu_0-2\mu_1+\mu_{2}(j) & \text{ if }|\beta|=2\text{ and }\beta\in \tilde{R}_j.
\end{cases} 
$$
Thus, $\lambda_\mathfrak B(\beta)$ does not depend on $\beta$ for $\beta \in \binom{[v]}{i}$ when $i\in\{0,1\}$ and depends only on $j$  if $|\beta|=2$ and $\beta\in \tilde{R}_j$.  
Hence, the hypergraph $\mathfrak H_{A}(v,k,a)$ 
is a $\operatorname{PBIBD}(v,k;\lambda_1,\ldots,\lambda_d)$ on $\mathcal{X}$.

Finally, we establish the expressions for the parameters $\lambda_1,\ldots,\lambda_d$.  
Since 
\begin{align*}
\mu_i&=\frac{1}{a-b}\left((-1)^{k} c_{i}-b\binom{v-i}{k}-(c-b)\nu_{i}\right)\quad\text{for $i\in \{0,1\}$},\\  
\mu_{2}(j)&=\frac{1}{a-b}\left((-1)^{k} c_{2}(j)-b\binom{v-2}{k}-(c-b)\nu_{2}(j)\right)\quad\text{for $j\in \{1,\dots,d\}$},   
\end{align*}
the parameters $\lambda_j$ are determined as 
\begin{align*}
\lambda_j&=\mu_0-2\mu_1+\mu_{2}(j)\\
    &=\frac{(-1)^k}{a-b}\left(c_0-2c_1+c_{2}(j)\right)-\frac{b}{a-b}\left( \binom{v}{k} - 2\binom{v-1}{k}+\binom{v-2}{k}\right)-\frac{c-b}{a-b}(\nu_{0}-2\nu_1+\nu_2(j))\\
    &=\frac{(-1)^k}{a-b}\left(c_0-2c_1+c_{2}(j)\right)-\frac{b}{a-b}\binom{v-2}{k-2}-\frac{c-b}{a-b}\eta_2(j).\qedhere
\end{align*}
\end{proof}

In the following two subsections, we apply Theorem~\ref{thm:pbibd2} to two different infinite families of matrices to obtain new constructions of PBIBDs.

\subsection{Signed hypercubes}\label{sec:signed}


Let $d$ be a positive integer. 
Define the matrix $S_d$ inductively as
\[
S_{d+1} := \begin{bmatrix}
S_d & I\\
I & -S_d
\end{bmatrix},
\]
where $S_0$ is the zero-matrix $O_1$ of order $1$.
By the definition of $S_d$, it is clear that $S_d^2 = d I$.
Since the trace of $S_d$ is $0$, we find that $\det(xI-S_d) = (x^2-d)^{2^{d-1}}$.
Furthermore, it is straightforward to verify that $D_{S_d}(4)=\{0,1,4\}$. 

Define $E_d = \{ \{x,y\} \in [2^d] \times [2^d] \; : \; (S_d)_{xy} = \pm 1 \}$.
The graph $Q_d := ([2^d],E_d)$ is called a hypercube.
As suggested by the title of this subsection, one can think of the matrix $S_d$ as a \emph{signed} adjacency matrix of a hypercube.
The matrix $S_d$ was used in the recent celebrated solution to the so-called \emph{sensitivity conjecture}~\cite{Huang}.
Let $d(x,y)$ denote the distance from the vertex $x$ to the vertex $y$ in the graph $Q_d$.
For each $i \in \{0,\dots,d\}$, define the relation $R_i$ by
\begin{align*}
R_i :=\{(x,y)\in [2^d]\times [2^d]  \; : \;  d(x,y)=i\}.
\end{align*}
Define $\mathcal H(d):=([2^d],\{R_i\}_{i=0}^d)$.
It is well-known that $\mathcal H(d)$ is a symmetric association scheme called the \emph{binary Hamming scheme}.
We now show that $\mathfrak H_{S_d}(2^d,4,4)$ is a PBIBD on $\mathcal H(d)$.


\begin{lemma}
\label{lem:signedcube}
    Let $d$ be a positive integer with $d\geqslant 2$ and set $A = S_d$. 
The hypergraph $\mathfrak H_{A}(2^d,4,4)$ is a $\operatorname{PBIBD}(2^d,4;d-1,1,0,\dots,0)$ on $\mathcal H(d)$.
\end{lemma}
\begin{proof}
Observe that $\det(S_d[\alpha])=4$ for $\alpha\in\binom{[2^d]}{4}$ if and only if $S_d[\alpha] = P^\transpose D \left ((J_2-I_2) \otimes \left[\begin{smallmatrix}
     1 & 1 \\
     1 & -1 \\
    \end{smallmatrix}\right] \right ) DP$ for some permutation matrix $P$ and diagonal matrix $D$ whose diagonal entries are $\pm 1$.
By the definition of $S_d$ and $Q_d$, it follows that each subset $\alpha \in \binom{[2^d]}{4}$ with $\det(S_d[\alpha])=4$ corresponds to an induced cycle of length $4$ in $Q_d$.
    
    The remainder of the proof follows from properties of the hypercube $Q_d$.
    We refer the reader to \cite[Section 9.2]{BCN} for the details and merely sketch the key takeaways below.
    One can prove by induction that each vertex belongs to precisely $\binom{d}{2}$ induced cycles of $Q_d$ of length $4$.
    It follows that $\mathfrak H_{A}(2^d,4,4)$ is a $1$-design.
    For $(x,y) \in R_1$, it is another straightforward induction to show that $\{x,y\}$ belongs to $d-1$ hyperedges of $\mathfrak H_{A}(2^d,4,4)$.
    For $(x,y) \in R_2$, there are precisely two vertices that are adjacent to both $x$ and $y$.
    Hence, $\{x,y\}$ belongs to precisely one hyperedge of $\mathfrak H_{A}(2^d,4,4)$.
    Finally, for $(x,y) \in R_i$ with $i \geqslant 3$, the distance from $x$ to $y$ in $Q_d$ is equal to $i$.
    Thus, $\{x,y\}$ does not belong to any hyperedge of $\mathfrak H_{A}(2^d,4,4)$.
\end{proof}



Recall that $\tilde{R}_i=\{\{x,y\}  \; : \;  (x,y)\in R_i\}$ for each $i\in\{0,1,\dots,d\}$. 
Now we are ready to state and prove the main theorem of this subsection.

\begin{theorem}\label{thm:signedcube}
Let $d$ be a positive integer with $d\geqslant 2$ and set $A = S_d$. 
Then, for each $a \in \{0,1\}$, the hypergraph $\mathfrak H_{A}(2^d,4,a)$
is a $\operatorname{PBIBD}(2^d,4;\lambda_1^{(a)},\ldots,\lambda_d^{(a)})$ on $\mathcal H(d)$ where 
\begin{align*}
    \lambda_1^{(a)} &= (1-a)\binom{2^d}{4}+(2a-1)(d^2+d(2^{d-1}-5-a)+3+a); \\
    \lambda_2^{(a)} &= (1-a)\binom{2^d}{4}+(2a-1)(d^2-3-a); \\
    \lambda_i^{(a)} &=(1-a)\binom{2^d}{4} + (2a-1)d^2, \text{ for each }i \in \{3,\dots,d\}.
\end{align*}
\end{theorem}
\begin{proof}
As mentioned above, $D_A(4)=\{0,1,4\}$.
Hence assumption (1) in Theorem~\ref{thm:pbibd2} is satisfied for $k=4$.  
Furthermore, by Lemma~\ref{lem:signedcube}, assumption (3) in Theorem~\ref{thm:pbibd2} is satisfied for $c = 4$.

Finally, we verify assumption (2) in Theorem~\ref{thm:pbibd2}.  
By Corollary~\ref{cor:Jacobi22}, we can write $\det(xI-A[\overline \alpha]) = (x^2-d)^{2^{d-1}-|\alpha|}\det(A[\alpha]+xI)$.
Next, observe that if $|\alpha| = 1$ then $A[\alpha] = O_1$.
Moreover, if $|\alpha| = 2$ then $A[\alpha] = \pm(J_2 - I_2)$ if $\alpha \in \tilde R_1$ and $A[\alpha] = O_2$ otherwise.
Thus,
\begin{align}
\det(xI-A[\overline \alpha])
&=\begin{cases}
x(x^2-d)^{2^{d-1}-1}&\text{ if }|\alpha|=1,\\ 
(x^2-1)(x^2-d)^{2^{d-1}-2}&\text{ if }|\alpha|=2\text{ and }\alpha\in\tilde{R}_1,
\\ 
x^2(x^2-d)^{2^{d-1}-2}&\text{ if }|\alpha|=2\text{ and }\alpha\in\tilde{R}_j \text{ with } j\geqslant 2.
\end{cases}\label{eq:signHam}
\end{align}
Therefore, $\det(xI-A[\overline \alpha])$ does not depend on the choice of $\alpha$ if $|\alpha|=1$ and depends only on $j$ where $\alpha\in\tilde{R}_j$ if $|\alpha|=2$. 
Thus, we may set $C_A \left (\binom{[2^d]}{i},4 \right ) = \{c_i\}$ for each $i \in \{0,1\}$ and $C_A \left (\tilde{R}_j,4 \right ) = \{c_2(j)\}$ for each $j \in \{1,\ldots,d\}$, where 
by expanding the expression \eqref{eq:signHam},  we obtain
\begin{align*}
c_i&=d^2\binom{2^{d-1}-i}{2}\text{ 
for }i\in\{0,1\},\quad
c_2(j)=d^2\binom{2^{d-1}-2}{2}+\delta_{j,1}d(2^{d-1}-2)\text{ 
for }j\in\{1,\ldots,d\}. 
\end{align*}
The expressions for $\lambda_1,\dots,\lambda_d$ can be obtained by substituting and simplifying the corresponding expressions in Theorem~\ref{thm:pbibd2}.    
\end{proof}



\subsection{Balanced generalised weighing matrices}\label{sec:bgw}

This subsection closely resembles Section~\ref{sec:signed}.
The difference is that, instead of the binary Hamming scheme $\mathcal H(d)$, we consider a certain 3-class symmetric association scheme and, instead of a signed hypercube, we consider a matrix derived from a so-called balanced generalised weighing matrix.

First, we need to give a couple of definitions.
The \textbf{incidence matrix} $N$ of a hypergraph $(\mathfrak X, \mathfrak B)$ is a $\{0,1\}$-matrix whose rows and columns are indexed by $\mathfrak X$ and $\mathfrak B$ respectively, where $N_{x,b} = 1$ if and only if $x \in b$.
A $2$-design $(\mathfrak X, \mathfrak B)$ is called \textbf{symmetric} if $|\mathfrak X| = |\mathfrak B|$. 
The incidence matrix $N$ of a symmetric $2$-$(v,k,\lambda)$ design satisfies $NN^\transpose=N^\transpose N=kI+\lambda (J-I)$ for $\lambda=\frac{k(k-1)}{v-1}$. 
A \textbf{balanced generalised weighing matrix} $W$ of \textbf{order} $v$ and \textbf{weight} $k$, denoted $\operatorname{BGW}(v,k,\lambda)$, is a $\{0,\pm 1\}$-matrix $W$ of order $v$ such that $WW^\transpose =kI$ and the matrix obtained from $W$ by replacing $-1$ with $1$ is the incidence matrix of a symmetric 2-$(v,k,\lambda)$ design.   
See \cite{JK,KPS} for infinite families of constructions of balanced generalised weighing matrices. 

Let $W$ be a $\operatorname{BGW}(v,k,\lambda)$ and define
$A:=\left [\begin{smallmatrix}
    O & W \\ W^\transpose &O 
\end{smallmatrix} \right ]$.
By the definition of $W$, it is clear that the matrix $A$ satisfies $A^2 = k I$.
Since the trace of $A$ is $0$, we find that $\det(xI-A) = (x^2-k)^{v}$.
Furthermore, it is straightforward to verify that $D_A(4)=\{0,1,4\}$. 
Define the $3$-class association scheme $\mathcal X_2 := ([2v],\{R_i\}_{i=0}^3)$ as follows \cite[Theorem~1.6.1]{BCN}.
Let $N$ be the $\{0,1\}$-matrix of order $v$ whose $(x,y)$-entry is $1$ if and only if the corresponding entry of $W$ is nonzero.
By the definition of $W$, the matrix $N$ is the incidence matrix of a $2$-$(v,k,\lambda$) design.
Define $A_0,\ldots,A_3$ by 
\begin{align*}
    A_0=I_{2v},\quad A_1=\begin{bmatrix} O & N \\
    N^\transpose & O \end{bmatrix}, \quad A_2=I_2\otimes (J_v-I_v), \quad A_3=\begin{bmatrix} O & J-N \\
    J-N^\transpose & O \end{bmatrix}.
\end{align*}
Now, for each $i \in \{0,1,2,3\}$ define $R_i$ by $(x,y) \in R_i$ if and only if the $(x,y)$ entry of $A_i$ is $1$.

 Next, we will show that $\mathfrak H_{A}(2v,4,4)$ is a PBIBD on $\mathcal X_2$.

\begin{lemma}
\label{lem:bgw}
    Let $v\geqslant 2$,  $W$ be a $\operatorname{BGW}(v,k,\lambda)$, and $A=\left [\begin{smallmatrix}
    O & W \\ W^\transpose &O 
\end{smallmatrix} \right ]$. 
The hypergraph $\mathfrak H_{A}(2v,4,4)$ is a $\operatorname{PBIBD}(2v,4;\frac{\lambda(k-1)}{2},\frac{\lambda^2}{4},0)$ on $\mathcal X_2$.
\end{lemma}
\begin{proof}
    For each $\alpha \in\binom{[2v]}{4}$, one can verify that $\det(A[\alpha])=4$ 
if and only if there exists a permutation matrix $P$ and a diagonal matrix $D$ whose diagonal entries are $\pm 1$ such that $A[\alpha] = P^\transpose D \left ((J_2-I_2) \otimes \left[\begin{smallmatrix}
     1 & 1 \\
     1 & -1 \\
    \end{smallmatrix}\right] \right ) DP$.
    Let $x \in [2v]$.
    We will show that the number of $4$-subsets $\alpha \subset [2v]$ such that $\det(A[\alpha])=4$ and $x \in \alpha$ is $(v-1)\lambda^2 /4$.
    By swapping the first half and second half of the elements of $[2v]$, if necessary, we can assume that $x \in [v]$.

    Suppose $\alpha \in\binom{[2v]}{4}$ and $x \in \alpha$.
    By the above, each subset $\alpha \in \binom{[2v]}{4}$ contains precisely two elements from $[v]$, i.e., $x$ and $y$ (say).
    There are $v-1$ ways to choose the element $y$.
    The remaining two elements $w$ and $z$ (say) of $\alpha$ must satisfy $A_{x,w} A_{y,w} A_{x,z} A_{y,z} = -1$.
    It follows from the definition of $A$ that there are $\lambda^2/4$ different ways to choose $w$ and $z$.
    Hence, $\mathfrak H_{A}(2v,4,4)$ is a $1$-design.

    Next, for fixed distinct $x \in [2v]$ and $y\in [2v]$, we define $K(x,y)$ to be the set of subsets $\alpha \in \binom{[2v] \backslash \{x,y\}}{2}$ satisfying $\det(A[\alpha \cup \{x,y\}])=4$.
    A similar argument to that of the above allows us to assume that $x \in [v]$.
    Moreover, if $y \in [v]$ then $\lambda_2 = |K(x,y)| = \lambda^2/4$.

    Assume that $(x,y) \in R_1$.
    We can count the number of $\{w,z\} \in K(x,y)$ as follows: $(k-1)$ choices for $w$ (since $A_{x,w}$ must be nonzero) then after choosing $w$, there are $\lambda/2$ choices for $z$.
    Hence, $\lambda_1 = |K(x,y)| = \lambda(k-1)/2$.
    Assume that $(x,y) \in R_3$.
    In this case, we require $A_{x,y} A_{w,y} A_{x,z} A_{w,z} = -1$ for some $w$ and $z$ but $A_{x,y} = 0$.
    Hence, $\lambda_3 = |K(x,y)| = 0$.
\end{proof}

Recall that $\tilde{R}_i=\{\{x,y\}  \; : \;  (x,y)\in R_i\}$ for each $i\in\{0,1,2,3\}$.
Now we are ready to state and prove the main theorem of this subsection.

\begin{theorem}\label{thm:bgw}
Let $v\geqslant 2$, $W$ be a $\operatorname{BGW}(v,d,\lambda)$, and $A=\left [\begin{smallmatrix}
    O & W \\ W^\transpose &O 
\end{smallmatrix} \right ]$. 
Then, for each $a\in\{0,1\}$, the hypergraph $\mathfrak H_{A}(2v,4,a)$
is a $\operatorname{PBIBD}(2v,4;\lambda_1^{(a)},\lambda_2^{(a)},\lambda_3^{(a)})$ on 
$\mathcal X_2$, where
\begin{align*}
    \lambda_1^{(a)} &= (1-a)\left (\binom{2v}{4}-\frac{\lambda(k-1)}{2} \right)+(2a-1)(d^2+d(v-2)-2\lambda(d-1)); \\
    \lambda_2^{(a)} &= (1-a)\left (\binom{2v}{4}-\frac{\lambda^2}{4}\right )+(2a-1)(d^2-\lambda^2); \\
    \lambda_3^{(a)} &= (1-a)\binom{2v}{4}+(2a-1)d^2.
\end{align*}
\end{theorem}

The proof is similar to that of Theorem~\ref{thm:signedcube}.

\begin{proof}
As mentioned above, $D_A(4)=\{0,1,4\}$.
Hence assumption (1) in Theorem~\ref{thm:pbibd2} is satisfied for $k=4$.  
Furthermore, by Lemma~\ref{lem:bgw}, assumption (3) in Theorem~\ref{thm:pbibd2} is satisfied for $c = 4$.

Finally, we verify assumption (2) in Theorem~\ref{thm:pbibd2}.  
By Corollary~\ref{cor:Jacobi22}, we can write $\det(xI-A[\overline \alpha]) = (x^2-d)^{v-|\alpha|}\det(A[\alpha]+xI)$.
Next, observe that if $|\alpha| = 1$ then $A[\alpha] = O_1$.
Moreover, if $|\alpha| = 2$ then $A[\alpha] = \pm(J_2 - I_2)$ if $\alpha \in \tilde R_1$ and $A[\alpha] = O_2$ otherwise.
Thus,
\begin{align}
\det(xI-A[\overline \alpha])
&=\begin{cases}
x(x^2-d)^{v-1}&\text{ if }|\alpha|=1,\\ 
(x^2-1)(x^2-d)^{v-2}&\text{ if }|\alpha|=2\text{ and }\alpha\in\tilde{R}_1,
\\ 
x^2(x^2-d)^{v-2}&\text{ if }|\alpha|=2\text{ and }\alpha\in\tilde{R}_2\cup \tilde{R}_3.
\end{cases}\label{eq:bgw}
\end{align}  
Therefore, $\det(xI-A[\overline \alpha])$ does not depend on the choice of $\alpha$ if $|\alpha|=1$ and depends only on $j$ where $\alpha\in\tilde{R}_j$ if $|\alpha|=2$. 
Thus, we may set $C_A \left (\binom{[2^d]}{i},4 \right ) = \{c_i\}$ for each $i \in \{0,1\}$ and $C_A \left (\tilde{R}_j,4 \right ) = \{c_2(j)\}$ for each $j \in \{1,2,3\}$, where 
by inspecting \eqref{eq:bgw}, we find
\begin{align*}
c_i&=d^2\binom{v-i}{2}\text{ 
for }i\in\{0,1\},\quad 
c_2(j)=d^2\binom{v-2}{2}+\delta_{j,1}d(v-2)\text{ 
for }j\in\{1,2,3\}. 
\end{align*}
The expressions for $\lambda_1$, $\lambda_2$, and $\lambda_3$ can be obtained by substituting and simplifying the corresponding expressions in Theorem~\ref{thm:pbibd2}. 
\end{proof}

\section{Regular pairwise balanced designs}\label{sec:pbd}

Let $v$ and $\lambda$ be positive integers, and let $K$ be a set of positive integers.  
A \textbf{pairwise balanced design}, denoted $\operatorname{PBD}(v,K,\lambda)$, is a $v$-vertex hypergraph $(\mathfrak X,\mathfrak B)$, where the cardinality of each hyperedge belongs to $K$ and every $2$-subset of $\mathfrak X$ is contained in exactly $\lambda$ hyperedges of $\mathfrak B$ \cite[Sections 7, 8]{Stinson}. 
A PBD is called \textbf{regular} if it is a $1$-design. 
Note that a regular $\operatorname{PBD}(v,\{k\},\lambda)$ is a $2$-$(v,k,\lambda)$ design. 
One can therefore think of a regular pairwise balanced design as generalisation of a $2$-design.

\subsection{Pairwise balanced designs from Hadamard matrices}

In this section, we provide an infinite family of regular pairwise balanced designs that can be obtained from a Hadamard matrix. 

Let $H$ be a Hadamard matrix of order $v$, and define $A=\left [\begin{smallmatrix}I & H \\ H^\transpose & I \end{smallmatrix} \right ]$.  
Clearly, the matrix $A$ satisfies $A^2-2A -(v-1)I= O$.
It follows that the characteristic polynomial $\det(xI - A) = (x^2-2x-v+1)^v$.
Furthermore, it is straightforward to check that $D_A(3)=\{-1,1\}$, $D_A(4)=\{-3,-2,1\}$. 


Let $A_0=I_{2v}$, $A_1=I_2\otimes (J_v-I_v)$, and $A_2=(J_2-I_2)\otimes J_v$.
For each $i \in \{0,1,2\}$ define $R_i$ by $(x,y) \in R_i$ if and only if the $(x,y)$ entry of $A_i$ is $1$.
In the same spirit as in Section~\ref{sec:signed} and Section~\ref{sec:bgw}, we will show that $\mathfrak H_{A}(2v,4,-3)$ is a PBIBD on $\mathcal X_3$.
The proof of Lemma~\ref{lem:hmpbd} below is similar to that of Lemma~\ref{lem:bgw}, and we invite the reader to modify the proof accordingly.

\begin{lemma}
\label{lem:hmpbd}
Let $H$ be a Hadamard matrix of order $v$ and $A=\left [\begin{smallmatrix}
    I & H \\ H^\transpose & I 
\end{smallmatrix} \right ]$. 
The hypergraph $\mathfrak H_{A}(2v,4,-3)$ is a $\operatorname{PBIBD}(2v,4;\frac{v(v-2)}{4},\frac{(v-1)(v-2)}{2})$ on $\mathcal X_{3}$.
\end{lemma}

Recall that $\tilde{R}_i=\{\{x,y\}  \; : \;  (x,y)\in R_i\}$ for each $i\in\{0,1,2\}$. 
We now construct a regular PBD from a Hadamard matrix. 
The method is based on constructions of PBIBDs in Section~\ref{sec:pbibd}.  

\begin{theorem}\label{thm:hmpbd}
Let $H$ be a Hadamard matrix of order $v$ and $A=\left [\begin{smallmatrix}
    I & H \\ H^\transpose & I 
\end{smallmatrix} \right ]$. 
Then the hypergraph $([2v],\mathfrak B_{A}(2v,3,-1)\cup \mathfrak B_{A}(2v,4,-2))$ is a regular $\operatorname{PBD}(2v,\{3,4\},v^2-v)$. 
\end{theorem}
\begin{proof}
As mentioned above, $D_A(3)=\{-1,1\}$ and $D_A(4)=\{-3,-2,1\}$.
Hence, assumption (1) in Theorem~\ref{thm:pbibd} is satisfied for $k=3$ and assumption (1) in Theorem~\ref{thm:pbibd2} is satisfied for $k=4$.   
Furthermore, by Lemma~\ref{lem:hmpbd}, assumption (3) in Theorem~\ref{thm:pbibd2} is satisfied for $c = -3$.

It remains to verify assumption (2) in Theorem~\ref{thm:pbibd} and Theorem~\ref{thm:pbibd2}.  
By Corollary~\ref{cor:Jacobi22}, we can write $\det(xI-A[\overline \alpha]) = (x^2-2x-v+1)^{v-|\alpha|}\det(A[\alpha]+xI)$.
Next, observe that if $|\alpha| = 1$ then $A[\alpha] = I_1$.
Moreover, if $|\alpha| = 2$ then $A[\alpha] = I_2$ if $\alpha \in \tilde R_1$ and $A[\alpha] = I_2\pm(J_2 - I_2)$ otherwise.
Thus,
\begin{align*}
\det(xI-A[\overline \alpha])
&=\begin{cases}
(x+1)(x^2-2x-v+1)^{v-1}&\text{ if }|\alpha|=1,\\ 
(x+1)^2(x^2-2x-v+1)^{v-2}&\text{ if }|\alpha|=2\text{ and }\alpha\in\tilde{R}_1,
\\ 
(x+2)(x-1)(x^2-2x-v+1)^{v-2}&\text{ if }|\alpha|=2\text{ and }\alpha\in\tilde{R}_2.
\end{cases}\label{eq:hmpbd}
\end{align*}
Therefore, $\det(xI-A[\overline \alpha])$ does not depend on the choice of $\alpha$ if $|\alpha|=1$ and depends only on $j$ where $\alpha\in\tilde{R}_j$ if $|\alpha|=2$. 

By Theorem~\ref{thm:pbibd}, 
the hypergraph $\mathfrak H_A(2v,3,-1)$ is a $\operatorname{PBIBD}(2v,3;\lambda_1,\lambda_2)$ on $\mathcal X_3$ with 
$\lambda_1=v$ and $\lambda_2=2(v-1)$. 
By Theorem~\ref{thm:pbibd2}  
 $\mathfrak H_A(2v,4,-2)$ is a $\operatorname{PBIBD}(2v,4;\lambda'_1,\lambda'_2)$ on $\mathcal X_3$ with 
$\lambda'_1=v(v-2)$ and $\lambda'_2=(v-1)(v-2)$. 

Since  
$
\lambda_1+\lambda'_1=\lambda_2+\lambda'_2=v^2-v
$, 
the pair $([2v],\mathfrak B_A(2v,3,-1)\cup\mathfrak B_A(2v,4,-2))$ is a regular $\operatorname{PBD}(2v,\{3,4\},v^2-v)$. 
\end{proof}

\subsection{Regular pairwise balanced designs from strongly regular graphs}

In this subsection, we provide more examples of regular pairwise balanced designs obtained from strongly regular graphs of small order.
The examples given here result from case-by-case computation and motivate further study of when regular pairwise balanced designs can be obtained from other combinatorial designs.

A $v$-vertex graph $\Gamma$ with adjacency matrix $A$ is called \textbf{strongly regular} if 
\[
A^2 = \kappa I + a A + c (J-I-A),
\]
for some integers $\kappa$, $a$, and $c$.
The tuple $(v,\kappa,a,c)$ is referred to as the \textbf{parameters} of $\Gamma$. 
Strongly regular graphs correspond to symmetric $2$-class association schemes in the sense that $\{I,A,J-I-A\}$ forms such an association scheme where $A$ is the adjacency matrix of a strongly regular graph.
Therefore, we have already seen that strongly regular graphs can be used to construct 
PBIBDs in Example~\ref{ex:symassPBIBD1} and Example~\ref{ex:symassPBIBD2}.
Now, we show that regular pairwise balanced designs can also be obtained from strongly regular graphs.

For $K\subset \{3,4,5\}$ and each $v$-vertex strongly regular graph with $v\leqslant 27$ with adjacency matrix $A$, we list all regular PBDs $([v], \mathfrak B)$ where 
\[
\mathfrak B = \bigcup_{k\in K}\mathfrak B_A(v,k,a_k)
\]
and $a_k\in D_A(k)$ for each $k \in K$.
This list is presented in Table~\ref{tab:my_label}, with the exception of regular PBDs of the form $([v],\mathfrak B_1\cup \mathfrak B_2)$ where each $([v],\mathfrak B_i)$ is a regular PBD, which are omitted. 
Indeed, if $([v],\mathfrak B_1)$ is a regular $\operatorname{PBD}(v,K_1,\lambda)$ and $([v],\mathfrak B_2)$ is a regular $\operatorname{PBD}(v,K_2,\lambda)$ then clearly $([v],\mathfrak B_1\cup \mathfrak B_2)$ is a regular $\operatorname{PBD}(v,K_1 \cup K_2,\lambda)$.
The case when $|K|=1$, however, is included in Table~\ref{tab:my_label}. 
\begin{table}[htb]
    \centering
    \begin{tabular}{c|c|c}
        $(v,\kappa,a,c)$ & regular $\operatorname{PBD}(v,K,\lambda)$  & Hyperedges $\mathfrak B$ \\
        \hline
        $(9,4,1,2)$ & $(9,\{3,5\},9)$ & $\mathfrak B_A(9,3,0)\cup \mathfrak B_A(9,5,-4)$ \\            
        $(9,4,1,2)$ & $(9,\{3,5\},33)$ & $\mathfrak B_A(9,3,2)\cup \mathfrak B_A(9,5,0)$ \\   
        $(10,3,0,1)$ & $(10,\{3\},8)$ & $\mathfrak B_A(10,3,0)$ \\
        $(13,6,2,3)$ & $(13,\{5\},5)$ & $\mathfrak B_A(13,5,2)$ \\
        $(13,6,2,3)$ & $(13,\{3,5\},15)$ & $\mathfrak B_A(13,3,0)\cup \mathfrak B_A(13,5,-4)$ \\
        $(15,8,4,4)$ & $(15,\{4\},30)$ & $\mathfrak B_A(15,4,0)$ \\
        $(15,8,4,4)$ & $(15,\{3,4,5\},21)$ & $\mathfrak B_A(15,3,0)\cup\mathfrak B_A(15,4,-3)\cup\mathfrak B_A(15,5,-4)$ \\
        $(15,8,4,4)$ & $(15,\{3,4,5\},13)$ & $\mathfrak B_A(15,3,0)\cup\mathfrak B_A(15,4,-3)\cup\mathfrak B_A(15,5,4)$ \\
        $(16,5,0,2)$ & $(16,\{3\},14)$ & $\mathfrak B_A(16,3,0)$ \\
        $(16,6,2,2)$ & $(16,\{4,5\},91)$ & $\mathfrak B_A(16,4,0)\cup \mathfrak B_A(16,5,-2)$ \\
        $(16,9,4,6)$ & $(16,\{4\},45)$ & $\mathfrak B_A(16,4,1)$ \\
        $(16,9,4,6)$ & $(16,\{5\},12)$ & $\mathfrak B_A(16,5,-4)$ \\
        $(16,10,6,6)$ & $(16,\{4,5\},51)$ & $\mathfrak B_A(16,4,0)\cup \mathfrak B_A(16,5,2)$ \\
        $(17,8,3,4)$ & $(17,\{5\},20)$ & $\mathfrak B_A(17,5,2)$ \\
        $(21,10,5,4)$ & $(21,\{5\},12)$ & $\mathfrak B_A(21,5,2)$ \\
        $(21,10,3,6)$ & $(21,\{5\},12)$ & $\mathfrak B_A(21,5,2)$ \\
        $(25,8,3,2)$ & $(25,\{3,4\},23)$ & $\mathfrak B_A(25,3,0)\cup \mathfrak B_A(25,4,-3)$ \\
        $(25,12,5,6)$  & $(25,\{5\},40)$ & $\mathfrak B_A(25,5,2)$ \\
        $(25,16,9,12)$ & $(25,\{3,5\},23)$ & $\mathfrak B_A(25,3,2)\cup \mathfrak B_A(25,5,-4)$ \\
        $(25,16,9,12)$ & $(25,\{3,4,5\},59)$ & $\mathfrak B_A(25,3,0)\cup \mathfrak B_A(25,4,-3)\cup \mathfrak B_A(25,5,-4)$ \\
        $(25,16,9,12)$ & $(25,\{3,4,5\},164)$ & $\mathfrak B_A(25,3,0)\cup \mathfrak B_A(25,4,0)\cup \mathfrak B_A(25,5,4)$  \\
        $(27,16,10,8)$ & $(27,\{3,4,5\},165)$ & $\mathfrak B_A(27,3,0)\cup \mathfrak B_A(27,4,0)\cup \mathfrak B_A(27,5,-4)$ \\
    \end{tabular}
    \caption{Regular pairwise balanced designs from strongly regular graphs.}
    \label{tab:my_label}
\end{table}
The leftmost column of Table~\ref{tab:my_label} contains the parameters of the corresponding strongly regular graph.
Constructions of these graphs can be found in Brouwer and Van Maldegham's book~\cite{SRG}.
Except for the parameters $(16,6,2,2)$ and $(25,12,5,6)$, the parameters in Table~\ref{tab:my_label} correspond uniquely to a strongly regular graph.
There are two strongly regular graphs with parameters $(16,6,2,2)$.
Just one of these two graphs gives rise to a PBD.
This graph can be defined as follows: the vertices form a $4 \times 4$ grid where each vertex is adjacent to all vertices in the same row and all vertices in the same column.
There are 15 strongly regular graphs with parameters $(25,12,5,6)$.
The one that gives rise to a PBD is the Paley digraph of order $25$ (recall the definition of a Paley digraph from Section~\ref{sec:des}).

One can read Table~\ref{tab:my_label} as follows.
Given the adjacency matrix $A$ of the $v$-vertex strongly regular graph corresponding to the parameters in the leftmost column, the hypergraph $([v],\mathfrak B)$, where $\mathfrak B$ is given by the rightmost column is a regular $\operatorname{PBD}(v,K,\lambda)$ corresponding to the middle column.
Note that $([10],\mathfrak B_{A_{10}}(10,3,0))$ and $([16],\mathfrak B_{A_{16}}(16,3,0))$ are trivial $2$-designs, where $A_{10}$ and $A_{16}$ are the adjacency matrices of strongly regular graphs with parameters $(10,3,0,1)$ and $(16,5,0,2)$, respectively.

\section{Concluding remarks and open problems}
\label{sec:open}
We conclude with some open questions which have naturally emerged throughout the course of our study.
Motivated by Example~\ref{ex:skew3}, we begin with a well-known open question originally due to \cite{S78}. 

\begin{question}
For which integers $n$ does there exist a skew-symmetric conference matrix of order $4n$? 
\end{question}

In Section~\ref{sec:des}, we were able to find examples of matrices $A$ of order $v$ and positive integers $k$ such that for $a \in D_A(k)$ the hypergraph $\mathfrak H_A(v,k,a)$ is a $t$-design where $t \leqslant 3$.

\begin{question}
For which $t\geqslant 4$ does there exist a matrix $A$ of order $v$ and a positive integer $k$ such that $\mathfrak H_A(v,k,a)$ is a $t$-design for some $a \in D_A(k)$?
\end{question}

In every example we have seen so far, the hyperedge cardinality $k \in \{3,4,5\}$.
Thus, the following question naturally arises.

\begin{question}
Does there exist a matrix $A$ of order $v$ and a positive integer $k \geqslant 6$ such that $\mathfrak H_A(v,k,a)$ is a $t$-design for some $a \in D_A(k)$ and $t \geqslant 1$?
\end{question}



Symmetric designs are extremal $2$-designs in the sense that, by Fisher's inequality~\cite{fisher}, any $2$-design $(\mathfrak X,\mathfrak B)$ must satisfy $|\mathfrak X| \leqslant |\mathfrak B|$.
So far, we are not aware of a matrix $A$ of order $v$ such that $\mathfrak H_A(v,k,a)$ is a symmetric design for some $a\in D_A(k)$ and some positive integer $k$.

\begin{question}
Does there exist a square matrix $A$ of order $v$ and a positive integer $k$ such that $\mathfrak H_A(v,k,a)$ is a symmetric design for some $a\in D_A(k)$? 
\end{question}

The Seidel matrices in Example~\ref{ex:ETF} can be generalised to the Hermitian matrices with similar properties, known as $q$-Seidel matrices~\cite{GW} or complex equiangular tight frames~\cite{etf}.
We give one example below, however, it does not satisfy the assumption of Theorem~\ref{thm:des}.

Set $X = \left [ \begin{smallmatrix}
    0 & 1 \\ 1 & 0
\end{smallmatrix} \right ]$ and $Y = \left [ \begin{smallmatrix}
    1 & 0 \\ 0 & -1
\end{smallmatrix} \right ]$.
For each binary vector $\mathbf k = (k_1,k_2,k_3,k_4,k_5,k_6) \in \{0,1\}^6$, define the matrix $T_{\mathbf k}$ as
\[
T_{\mathbf k} := X^{k_1}Y^{k_2} \otimes X^{k_3}Y^{k_4} \otimes X^{k_5}Y^{k_6}. 
\]
Set $\mathbf v = \frac{1}{2\sqrt{3}}(-1+2\sqrt{-1},1,1,1,1,1,1,1)$.
The lines spanned by the unit vectors $\mathbf v T_{\mathbf k}$ for $\mathbf k \in \{0,1\}^6$ are called \textbf{Hoggar's lines}~\cite{H,JW}.
Let $G$ be the Gram matrix of the vectors $\mathbf v T_{\mathbf k}$. 

\begin{example}
\label{ex:hoggar}
Set $S=3(G-I_{64})$. 
Then $S$ is a Hermitian $\{0,\pm1,\pm i\}$-matrix with zero diagonal entries and non-zero off-diagonal entries, and
\begin{align*}
D_S(3)&=\{-2,0,2\},\\
D_S(4)&=\{-3,1,5,9\}.
\end{align*}
Note that the assumptions of Theorem~\ref{thm:des} are not satisfied with $S$ and $k=3$ or $k=4$.
However, to each row $(k,a,\lambda)$ of Table~\ref{tab:hoggar}, the hypergraph $\mathfrak H_S(64,k,a)$ is a $2$-$(64,k,\lambda)$ design. 
\begin{table}[h!]
    \centering
    \begin{tabular}{c|c|c}
         $k$ & $a$ & $\lambda$ \\
        \hline
        3 & $-2$ & 6\\
        3 & 0 & 32 \\
        3 & 2 & 24\\
        4 & $-3$ & 123\\
        4 & 1 & 1008\\
        4 & 5 & 648\\
        4 & 9 & 112
    \end{tabular}
    \caption{The $2$-$(64,k,\lambda)$ designs $\mathfrak H_{S}(64,k,a)$.}
    \label{tab:hoggar}
\end{table}
\end{example}

To understand how Example~\ref{ex:hoggar} fits into a more general framework is a topic requiring further investigation.  

\begin{question}
Find simple sufficient conditions on a complex $v \times v$ matrix $A$ satisfying $|D_A(k)| = 3$ that guarantees that, for $a \in D_A(k)$, the hypergraph $\mathfrak H_A(v,k,a)$ is a $t$-design.
\end{question}


Let $A$ be the adjacency matrix of a doubly regular tournament of order $v \geqslant 5$. 
It was shown in Example~\ref{ex:doubly} that for $k\in \{3,4\}$ and $a\in D_A(k)$, the hypergraph $\mathfrak H_A(v,k,a)$ is a $2$-design. 
Now, we consider the case for $k=5$. 
It is straightforward to verify that $D_A(5)=\{0,1,2,3\}$. 
An exhaustive list of doubly regular tournaments of order at most $27$ is available from the web-page~\cite{M}. 
Using a computer, we find that, for each doubly regular tournament of order $4n+3 \leqslant 27$ with adjacency matrix $A$, the hypergraph $\mathfrak H_A(4n+3,5,0)$ is a $2$-design. 

\begin{question}
    For any doubly regular tournament on $v$ vertices with adjacency matrix $A$, is $\mathfrak H_A(v,5,0)$ a $2$-design?
\end{question}

A Paley digraph of order $4n+3$ (defined in Section~\ref{sec:des}) is a doubly regular tournament. 
Moreover, for each Paley digraph of order $4n+3 \leqslant 43$ with adjacency matrix $A$, the hypergraph $\mathfrak H_A(4n+3,5,a)$ is a $2$-design for any $a\in D_A(5)$.  


\begin{question}
Suppose $A$ is an adjacency matrix of a Paley digraph of order $4n+3$.
Is $\mathfrak H_A(4n+3,5,a)$ a $2$-design for any $a\in D_A(5)$?
\end{question}


Let $v$, $\lambda$, and $t$ be positive integers, and let $K$ be a set of positive integers.  
A \textbf{$t$-wise balanced design}, denoted $t$-BD$(v,K,\lambda)$, is a $v$-vertex hypergraph $(\mathfrak X,\mathfrak B)$, where the cardinality of each hyperedge belongs to $K$ and every $t$-subset of $\mathfrak X$ is contained in exactly $\lambda$ hyperedges of $\mathfrak B$~\cite{KK}. 
A $t$-BD is called \textbf{regular} if it is a $s$-BD for any $1\leqslant s<t$. 

In Theorem~\ref{thm:hmpbd}, we obtained an infinite family of regular $2$-BDs.


\begin{question}
Let $t\geqslant 3$.
Does there exist a finite set $K$ of positive integers and an infinite family of matrices $A$ such that 
\[
\left ([v],\bigcup_{k\in K}\mathfrak B_A(v,k,a_k) \right )
\]
is a regular $t$-BD,   
where $a_k\in D_A(k)$ for each $k \in K$ and $v$ is the order of $A$?
\end{question}


\section{Acknowledgement}

We are grateful to the referees for their insightful comments and suggestions, which have contributed to the improvement of this paper.

\end{document}